\documentclass[12pt]{amsart}
\usepackage[english]{babel}
\usepackage[margin=1.2in]{geometry}

\usepackage{xcolor}
\usepackage[normalem]{ulem} 
\usepackage{cancel} 
\usepackage{amsfonts}
\usepackage{amssymb}
\usepackage{amsmath}
\usepackage{amsthm}
\usepackage{mathtools}
\usepackage{diffcoeff}
\usepackage{interval}
\intervalconfig{soft open fences}

\makeatletter
\@namedef{subjclassname@2020}{\textup{2020} Mathematics Subject Classification}
\makeatother

\usepackage{enumitem}
\setlist[enumerate]{label=\((\roman*)\)}

\usepackage{hyperref}
\usepackage[capitalize,nameinlink]{cleveref}
\crefformat{equation}{(#2#1#3)}
\crefname{subsection}{Subsection}{Subsections}

\theoremstyle{plain}
\newtheorem{theorem}{Theorem}[section]
\newtheorem{lemma}[theorem]{Lemma}
\newtheorem{proposition}[theorem]{Proposition}
\newtheorem{corollary}[theorem]{Corollary}
\newtheorem{conjecture}[theorem]{Conjecture}

\theoremstyle{definition}

\newtheorem{remark}[theorem]{Remark}

\newtheorem{example}[theorem]{Example}

\numberwithin{equation}{section}

\newcommand{\N}{\mathbb{N}}

\newcommand{\R}{\mathbb{R}}
\newcommand{\C}{\mathbb{C}}
\NewDocumentCommand{\RR}{O{\Dim}}{\R^{#1}}
\DeclareMathOperator{\Supp}{supp}
\NewDocumentCommand{\Int}{m}{\operatorname{int}#1}
\NewDocumentCommand{\GGL}{m}{\operatorname{GL}(#1)}
\renewcommand{\epsilon}{\varepsilon}
\newcommand{\Molf}{\chi}
\newcommand{\Cdot}{\, \cdot \,}                                    
\newcommand{\Lapl}{\Delta}                                         
\NewDocumentCommand{\LaplP}{s}%
{\IfBooleanTF{#1}{\Id - \Lapl}{(\Id - \Lapl)}}                     
\newcommand{\Mnf}{M}                                               
\newcommand{\Dim}{n}                                               
\newcommand{\Vecs}{E}                                              
\NewDocumentCommand{\Frame}{O{X}}{\mathbf{#1}}                     
\NewDocumentCommand{\Bnds}{O{\Vecs}}{\mathfrak{B}(#1)}             
\NewDocumentCommand{\Weight}{s}{\IfBooleanTF{#1}{\tau}{\sigma}}    
\NewDocumentCommand{\WeightExp}{O{\varphi}}{#1}
\NewDocumentCommand{\WeightYoung}{O{\WeightExp}}{{#1}^{\ast}}
\NewDocumentCommand{\FT}{m}{\widehat{#1}}                          
\DeclareMathOperator{\FTO}{\mathcal{F}}                            
\DeclareMathOperator{\Id}{Id}
\DeclareMathOperator{\Trace}{Tr}
\DeclareMathOperator{\Span}{span}

\newcommand{\LieG}{G}                                              
\newcommand{\DLieG}{\widehat{G}}                                   
\newcommand{\Origin}{e}                                              

\NewDocumentCommand{\Cont}{}{C}                                    
\NewDocumentCommand{\Smooth}{O{\infty}}{C^{#1}}                    
\NewDocumentCommand{\BSmooth}{O{\infty}}{C_b^{#1}}                 
\NewDocumentCommand{\SDiff}{O{\Weight} m}{\mathcal{E}^{#1,#2}}     
\NewDocumentCommand{\UDiff}{O{\Weight}}{\mathcal{E}^{[#1]}}        
\NewDocumentCommand{\BDiff}{O{\Weight}}{\mathcal{E}^{(#1)}}        
\NewDocumentCommand{\RDiff}{O{\Weight}}{\mathcal{E}^{\{#1\}}}      
\NewDocumentCommand{\BUDiff}{O{\Weight}}{\mathcal{E}_b^{[#1]}}     
\NewDocumentCommand{\BBDiff}{O{\Weight}}{\mathcal{E}_b^{(#1)}}     
\NewDocumentCommand{\BRDiff}{O{\Weight}}{\mathcal{E}_b^{\{#1\}}}   

\NewDocumentCommand{\SmoothC}{}{s}                                 
\NewDocumentCommand{\SDiffC}{O{\Weight} m}{\SmoothC^{#1,#2}}       
\NewDocumentCommand{\UDiffC}{O{\Weight}}{\SmoothC^{[#1]}}          
\NewDocumentCommand{\BDiffC}{O{\Weight}}{\SmoothC^{(#1)}}          
\NewDocumentCommand{\RDiffC}{O{\Weight}}{\SmoothC^{\{#1\}}}        
\NewDocumentCommand{\BUDiffC}{O{\Weight}}{\SmoothC_b^{[#1]}}       
\NewDocumentCommand{\BBDiffC}{O{\Weight}}{\SmoothC_b^{(#1)}}       
\NewDocumentCommand{\BRDiffC}{O{\Weight}}{\SmoothC_b^{\{#1\}}}     

\NewDocumentCommand{\RAnalV}{}{\Vecs^{\omega}}                     
\NewDocumentCommand{\SmoothV}{O{\infty}}{\Vecs^{#1}}               
\NewDocumentCommand{\BSmoothV}{O{\infty}}{\Vecs^{#1}}              
\NewDocumentCommand{\UDiffV}{O{\Weight}}{\Vecs^{[#1]}}             
\NewDocumentCommand{\BDiffV}{O{\Weight}}{\Vecs^{(#1)}}             
\NewDocumentCommand{\RDiffV}{O{\Weight}}{\Vecs^{\{#1\}}}           
\NewDocumentCommand{\BUDiffV}{O{\Weight}}{\Vecs^{[#1]}}            
\NewDocumentCommand{\BBDiffV}{O{\Weight}}{\Vecs^{(#1)}}            
\NewDocumentCommand{\BRDiffV}{O{\Weight}}{\Vecs^{\{#1\}}}          

\newcommand{\Rep}{\pi}                                             
\newcommand{\LRep}{\pi_{L}}                                        
\newcommand{\Action}{\Pi}                                          
\NewDocumentCommand{\Orbit}{m}{\gamma_{#1}}                        

\NewDocumentCommand{\IrredRep}{}{\xi}                              
\NewDocumentCommand{\IrredRepC}{}{[\xi]}                           
\NewDocumentCommand{\IrredRepD}{}{d_{\IrredRep}}                   
\NewDocumentCommand{\IrredRepE}{O{\IrredRep}}{\lambda_{#1}}        
\NewDocumentCommand{\IrredRepS}{O{\IrredRep}}{\mathcal{H}_{#1}}    
\NewDocumentCommand{\MatCoefS}{O{\IrredRep}}{\mathtt{E}_{#1}}      
\NewDocumentCommand{\IrredRepEP}{s O{\IrredRep}}%
{\IfBooleanTF{#1}{1+\IrredRepE[#2]}{(1+\IrredRepE[#2])}}           

\NewDocumentCommand{\Eval}{m m}{\left<#1,#2\right>}                
\DeclareMathOperator{\csn}{csn}
\DeclarePairedDelimiter{\abs}{\lvert}{\rvert}
\DeclarePairedDelimiter{\norm}{\lVert}{\rVert}
\NewDocumentCommand{\HSnorm}{m}{\norm{#1}_{\text{HS}}}
\NewDocumentCommand{\pHSnorm}{O{p}}{#1_{\text{HS}}}
\NewDocumentCommand{\psupnorm}{O{p}}{#1_{\LieG}}

\title[Strong factorization of ultradifferentiable vectors]{Strong factorization of ultradifferentiable vectors associated with compact Lie group representations}

\author[A. Debrouwere]{Andreas Debrouwere}
\address{A. Debrouwere, Department of Mathematics and Data Science \\ Vrije Universiteit Brussel, Belgium\\ Pleinlaan 2 \\ 1050 Brussels \\ Belgium}
\email{Andreas.Debrouwere@vub.be}

\author[M. Huttener]{Michiel Huttener}
\address{M. Huttener, Department of Mathematics: Analysis, Logic and Discrete Mathematics\\ Ghent University\\ Krijgslaan 281\\ 9000 Ghent\\ Belgium}
\email{Michiel.Huttener@UGent.be}

\author[J. Vindas]{Jasson Vindas}
\address{J. Vindas, Department of Mathematics: Analysis, Logic and Discrete Mathematics\\ Ghent University\\ Krijgslaan 281\\ 9000 Ghent\\ Belgium}
\email{Jasson.Vindas@UGent.be}

\thanks {M.~Huttener and J.~Vindas were supported by the Research Foundation-Flanders through the FWO grant number G067621N. J.~Vindas also acknowledges support from the Ghent University grant number bof/baf/4y/2024/01/155.}


\begin{document}

\begin{abstract}

    We show a strong factorization theorem of Dixmier-Malliavin type for ultradifferentiable vectors associated with compact Lie group representations on sequentially complete locally convex Hausdorff spaces.
    In particular, this solves a conjecture by  Gimperlein et al. [J. Funct. Anal. 262 (2012), 667--681] for analytic vectors in the case of compact Lie groups.

\end{abstract}

\maketitle

\section{Introduction}

Let \(\mathcal{M}\) be a (left) module over a nonunital algebra \(\mathcal{A}\).
We say that it has the \emph{strong (weak) factorization property} if \(\mathcal{M} = \mathcal{A} \cdot \mathcal{M}\) (if \(\mathcal{M} = \Span(\mathcal{A} \cdot \mathcal{M})\)).
If \(\mathcal{M}\) is also a bornological space, then it is said to satisfy the \emph{bounded strong factorization property} if for any bounded subset \(\mathcal{B} \subset \mathcal{M}\), one can find \(a \in \mathcal{A}\) and a bounded subset \(\mathcal{B}' \subset \mathcal{M}\) such that \(\mathcal{B} = a \cdot \mathcal{B}'\).

Establishing factorization properties for convolution modules over function algebras is an important topic in harmonic analysis with a long tradition that goes back to works by Salem \cite{Salem39}, Rudin \cite{Rudin57}, and Cohen \cite{Cohen59}. 
In the case of smooth functions, the subject was greatly stimulated by Ehrenpreis' question \cite{ehrenpreis} on the factorization properties of the convolution algebra $\mathcal{D}(\R^d)$ of compactly supported smooth functions. 
We mention that $\mathcal{D}(\R^d)$ has the weak but not the strong factorization property when $d \geq 2$ \cite{rst,D-M}, while it does have the strong factorization property when $d = 1$ \cite{yul}.

In \cite{D-M}, Dixmier and Malliavin developed a general framework for studying factorization properties of smooth vectors associated with Lie group representations.
Let \(\Rep\) be a (strongly) continuous representation of a Lie group \(\LieG\) on a Fréchet space \(\Vecs\).
The space of smooth vectors \(\SmoothV\) of \(\Rep\) consists of those \(v \in \Vecs\) whose orbit \(x \mapsto \Rep(x)v\) is smooth.
The representation \(\Rep\) induces an action \(\Action\) of the convolution algebra of compactly supported smooth functions \(\mathcal{D}(\LieG)\) on \(\Vecs\) via
\begin{equation}
    \label{equ:action}
    \Action(\chi)v 
    = \int _{\LieG} \chi(x) \, \Rep(x) v \dl{x}
    , \qquad \chi \in \mathcal{D}(\LieG), \  v \in \Vecs
    ,
\end{equation}
which restricts to an action on \(\SmoothV\).
In such a way, \(\SmoothV\) becomes a module over \((\mathcal{D}(\LieG), \ast)\).
The Dixmier-Malliavin factorization theorem states that \(\SmoothV\) possesses the weak factorization property as a module over \((\mathcal{D}(\LieG), \ast)\).
In addition, if \(\LieG\) is a compact Lie group, Dixmier and Malliavin even established the bounded strong factorization property for \(\SmoothV\) as a module over \((\Smooth(\LieG), \ast)\).

More recently,   Gimperlein, Krötz, and Lienau \cite{G-K-L} generalized the Dixmier-Malliavin factorization theorem to analytic vectors \(\RAnalV\) associated with \(F\)-representations of Lie groups on Fréchet spaces.
Analytic vectors are essential tools in various aspects of Lie group representation theory, see for instance \cite{Harish-Chandra, G-K-S, Goodman1969, Nelson59}.
The space of analytic vectors has a natural \(\mathcal{A}(\LieG)\)-module structure, where \(\mathcal{A}(\LieG)\) is the space of real analytic functions with appropriate superexponential decay at infinity, such that \eqref{equ:action} remains well-defined for all \(\chi \in \mathcal{A}(\LieG)\) and \(v\in\RAnalV\).
They showed that \(\RAnalV\) always has the weak factorization property.
Furthermore, they conjectured:
\begin{conjecture}[{\cite[Conjecture 6.4]{G-K-L}}]
    \label{GKL conjecture}
    The space of analytic vectors associated with any \(F\)-representation \(\Rep\) of a Lie group \(\LieG\) on a Fréchet space \(\Vecs\) has the strong factorization property over \(\mathcal{A}(\LieG)\), that is, \(\RAnalV = \Action(\mathcal{A}(\LieG))\RAnalV \).
\end{conjecture}

So far, \cref{GKL conjecture} is only known to hold true in the Euclidean case \(\LieG = (\RR,+)\).
In fact, it was shown in \cite{DPV21} to be true for a fairly general class of representations, including \(F\)-representations and, more generally, proto-Banach representations \cite{Glo}.

The main aim of this article is to establish \cref{GKL conjecture} when \(\LieG\) is a compact Lie group; moreover, our results shall yield:

\begin{theorem}\label{th analytic} 
    Let \(\Rep\) be a  representation of a compact Lie group \(\LieG\) on a Fréchet space \(\Vecs\).
    Then, \(\RAnalV\) has the bounded strong factorization property with respect to the convolution algebra of real analytic functions on \(\LieG\).
\end{theorem} 

In fact, we will go beyond both Fréchet representations and real analyticity, and work with arbitrary representations on sequentially complete locally convex Hausdorff spaces and ultradifferentiable vectors.
Our considerations shall cover the case of analytic vectors (see \cref{rem:analytic-vectors-coincide}), but also the important instance of Gevrey vectors that were introduced and thoroughly investigated by Goodman and Wallach in \cite{Goodman1,Goodman2,GW}.
As a corollary of this general framework, we shall also recover the original Dixmier-Malliavin factorization theorem for compact Lie groups, providing a new approach to its proof.

The plan of this article is as follows.
In \cref{sec:ultradifferentiability}, we introduce vector-valued ultradifferentiable functions on real analytic manifolds via weight functions (in the Braun-Meise-Taylor setting \cite{BMT}), which will be subsequently used to define (bornologically) ultradifferentiable vectors associated with compact Lie group representations. 
These classes have a natural module structure over some convolution algebra of ultradifferentiable functions on the group.
Our main results and some of their important corollaries are then stated in \cref{sec:main-result}.
In preparation for the proofs, we discuss in \cref{sec:laplacian} a convenient basis of continuous seminorms for spaces of ultradifferentiable functions described in terms of the Laplace-Beltrami operator on the group.
\Cref{sec:fourier-coefficients} contains the technical crux of our solution to the considered factorization problems.
Using the Peter-Weyl Theorem, we derive a Fourier characterization of ultradifferentiable vector-valued functions, which combined with the identity \(\FT{f \ast g} = \FT f \circ \FT g\) then reduces the factorization of ultradifferentiable vectors to a division problem on group Fourier coefficients.

The proof of our main result is then given in \cref{sec:main-proof}.
Finally, in \cref{sec: factorization non-quasianalytic vectors}, we obtain an improvement of the factorization theorem for the case of non-quasianalytic vectors.

\section{Spaces of vector-valued ultradifferentiable functions}
\label{sec:ultradifferentiability}
In this section, we discuss vector-valued ultradifferentiable functions and recall their fundamental properties.
We first consider these classes on open subsets of \(\RR\).
They will then be generalized to manifolds, in particular to our setting, compact Lie groups.

\subsection{Weight functions}
By a \emph{weight function}, \cite{BMT,Rainer14} we mean a continuous increasing function \(\Weight \colon \rinterval{0}{\infty} \rightarrow \rinterval{0}{\infty}\) satisfying \(\Weight_{|\interval{0}{1}} \equiv 0\) and the following properties: 
\begin{itemize}
    \item[\((\alpha)\)] \(\Weight(2t) = O(\Weight(t))\).
    \item[\((\beta)\)] \(\Weight(t) = O(t)\).
    \item[\((\gamma)\)] \(\log t = o(\Weight(t))\).
    \item[\((\delta)\)]\(\WeightExp \colon \rinterval{0}{\infty} \rightarrow \rinterval{0}{\infty}\), \(\WeightExp(t) = \Weight(e^{t})\), is convex.
\end{itemize}
Here we use Landau's big $O$- and little $o$-symbols: 
Given two functions  \(w_1,w_2 \colon \rinterval{0}{\infty} \rightarrow \rinterval{0}{\infty}\) we write $w_1(t) = O(w_2(t))$  to indicate that there are $c, t_c > 0$ such that $w_1(t) \leq cw_2(t)$ for all $t \geq t_c$, while   $w_1(t) = o(w_2(t))$  means that the previous inequality holds for every $c >0$ and some $t_c>0$.
We will sometimes strengthen \((\beta)\) to 
\begin{itemize}
    \item[ \((\beta_0)\)] \(\Weight(t) = o(t)\).
\end{itemize}
Note that the non-quasianalyticity condition \(\displaystyle \int _{0}^{\infty}t^{-2}\Weight(t)\dl{t} < \infty\) from \cite{BMT} implies \((\beta_0)\).

\begin{example}\label{exa:gevrey-weight}
    The \emph{Gevrey weight of order \(s\)} is defined as \(\Weight_s(t) = \max \{0, t^s - 1 \}\).
    It is a weight function for \(0 < s \leq 1\) and additionally satisfies \((\beta_0)\) if \(s < 1\).
\end{example}

Throughout the rest of this article we fix a weight function \(\Weight\) and write \(\WeightExp(t) = \Weight(e^t)\) (cf.\ condition \((\delta)\) above).
We define the \emph{Young conjugate} \(\WeightYoung \colon \rinterval{0}{\infty} \rightarrow \rinterval{0}{\infty}\) of \(\WeightExp\) as \(\WeightYoung(t) = \sup_{u \geq 0} \{tu- \WeightExp(u)\}\).
The function \(\WeightYoung\) is increasing, convex, and satisfies \(\WeightYoung(0) = 0\), \((\WeightYoung)^* = \WeightExp\), and \(\WeightYoung(t)/t \nearrow \infty\) on \(\rinterval{0}{\infty}\).

\begin{example}\label{exa:gevrey-conjugate}
    We set \(\WeightExp_s(t) = \Weight_s(e^t)\) for the Gevrey weights \(\Weight_s\).
    Then,
    \[
        \exp(\WeightYoung_s(t)) 
        = e \left(\frac{1}{se}\right)^{\frac{t}{s}} t^{\frac{t}{s}}
        .
    \]
    Consequently, we have, for all \(h>0\),
    \[
        \exp \left (\frac{1}{h}\WeightExp_s^*(ht) \right) 
        = e^{\frac{1}{h}}\left(\frac{h}{se}\right)^{\frac{t}{s}} t^{\frac{t}{s}}
        .
    \]
\end{example}

We will need the following property of the Young conjugate \(\WeightYoung\) (cf.\ \cite[Lemma 5.7]{Rainer14}): 
For every \(h>0\) there are \(h'<h\) and \(C >0\) such that 
\begin{equation}
    \label{equ:estimate-weight-young}
    \frac{1}{h} \Weight(t)
    \leq \sup_{k \in \N} k \log(t) - \frac{1}{h'} \WeightYoung\left(kh'\right) + C
    , \qquad t > 0
    .
\end{equation}

\subsection{Ultradifferentiable classes on open subsets of \texorpdfstring{\(\RR\)}{Rⁿ}}
\label{sec:ultradifferentiability-euclidean}

Let \(\Theta \subseteq \RR\) be open and let \(\Vecs\) be a (complex) lcHs (= locally convex Hausdorff space).
We denote by \(\csn(\Vecs)\) the family of all continuous seminorms on \(\Vecs\).
Given \(p \in \csn(\Vecs)\), we set \(V_p = \{ v \in \Vecs \mid p(v) \leq 1 \}\) and write \(V^\circ_p\) for its polar set in \(\Vecs'\).
The bipolar theorem yields
\begin{equation}
    \label{equ:bipolar-theorem}
    p(v) 
    = \sup_{v' \in V^\circ_p} \abs{\Eval{v'}{v}}
    , \qquad v \in \Vecs
    .
\end{equation}
For  \(K \subseteq \Theta\) compact, $h >0$, and \(p \in \csn(\Vecs)\), we write
\[
    p_{K,\Weight,h}(f) 
    = \sup_{x \in K} \sup_{\alpha \in \N^\Dim} p(f^{(\alpha)}(x))\exp \left(-\frac{1}{h}\WeightYoung(h\abs{\alpha})\right), \qquad f \in C^\infty(\Theta).
\]
We set \(p_{K,\Weight,h} = \norm{\Cdot}_{K,\Weight,h}\) for \(\Vecs = \C\) and \(p = \abs{\Cdot}\). We define $ \BDiff(\Theta;\Vecs)$ as the space consisting of all \(f \in \Smooth(\Theta;\Vecs)\) such that  $p_{K,\Weight,h}(f) <\infty$ for all \(K \subseteq \Theta\) compact, $h >0$, and \(p \in \csn(\Vecs)\), and endow it with the Hausdorff locally convex topology generated by the system of seminorms \(\{ p_{K,\Weight,h} \mid K \subseteq \Theta \text{ compact}, h>0, p \in \csn(\Vecs) \}\).
Following an idea of Komatsu \cite{Komatsu3}, we define
$$
\RDiff(\Theta;\Vecs) = \varprojlim_{\Weight* = o(\Weight)} \BDiff[\Weight*](\Theta;\Vecs)
$$
as lcHs. 
See Lemma \ref{thm:inductive-description-roumieu} below for  a more standard inductive description of the space \(\RDiff(\Theta;\Vecs)\) in case \(\Vecs\) is a normed space.  
We call $\BDiff(\Theta;\Vecs)$ and $\RDiff(\Theta;\Vecs)$ the spaces of Beurling and Roumieu \(\Vecs\)-valued ultradifferentiable functions of class \(\Weight\), respectively. The notation \(\UDiff(\Theta;\Vecs)\) stands for either case and we simply write \(\UDiff(\Theta)\) instead of \(\UDiff(\Theta;\C)\); similar conventions will be employed for other spaces as well.

\begin{lemma}
    \label{thm:inductive-description-roumieu} 
    Let \((\Vecs,p)\) be a normed space.
    Then, \(f \in \Smooth(\Theta;\Vecs)\) belongs to \(\RDiff(\Theta;\Vecs)\) if and only if for every compact subset \(K \subseteq \Theta \) there is \(h >0\) such that \(p_{K,\Weight,h}(f)< \infty\).
    Moreover, \(B \subseteq \RDiff(\Theta;\Vecs)\) is bounded if and only if for every compact subset \( K \subseteq \Theta \) there is \(h >0\) such that \(\sup_{f \in B} p_{K,\Weight,h}(f)< \infty\).
\end{lemma}
\begin{proof}
    It is shown in \cite[Corollary 2]{DPV22} that 
    \begin{equation}
        \label{topC}
        \RDiff(\Theta) 
        = \varprojlim_{K \Subset \Theta} \varinjlim_{h \to \infty} \SDiff{h}(K)
    \end{equation}
    as lcHs. 
    Here, \(K\) runs over all topologically regular compact subsets of \(\Theta\) (i.e.\ \(\overline{\Int K} = K\)) and \(\SDiff{h}(K)\) denotes the Banach space consisting of all \(f \in \Smooth(K)\) such that \(\norm{f }_{K,\Weight,h}<\infty\).
    By \cite[p.\ 80, Corollary 7]{Bierstedt}, the \((LB)\)-space 
    \[
        \RDiff(K) = \varinjlim_{h \to \infty} \SDiff{h}(K)
    \]
    is regular for each \(K\) (i.e. every bounded set in \(\RDiff(K)\) is contained and bounded in \(\SDiff{h}(K)\) for some \(h >0\)).
    Hence, the result for \(\Vecs = \C\) follows from the topological identity \eqref{topC}.
    Next, we show the result for general normed spaces \(\Vecs\).
    It suffices to show that a set \(B \subseteq \Smooth(\Theta;\Vecs)\) is contained and bounded in \(\RDiff(\Theta;\Vecs)\) if and only if
    \begin{equation}
        \label{equ:boundedness-condition}
        \forall K \subseteq \Theta \text{ compact } \exists h >0 
        \colon \sup_{f \in B} p_{K,\Weight,h}(f) < \infty
        .
    \end{equation}
    The first statement then follows by applying this to the singleton set \(B = \{f\}\).
    It is clear that \(B \subseteq \Smooth(\Theta;\Vecs)\) is contained and bounded in \(\RDiff(\Theta;\Vecs)\) if \(B\) satisfies \eqref{equ:boundedness-condition}.
    Conversely, let \(B \subseteq \Smooth(\Theta;\Vecs)\) be bounded.
    For \(v' \in \Vecs'\) and \(f \colon \Theta \to \Vecs\), we define the mapping \(\Eval{v'}{f} \colon \Theta \to \C, \, x \mapsto \Eval{v'}{f(x)}\).
    If \(f \in \Smooth(\Theta;E)\), then \(\Eval{v'}{f} \in \Smooth(\Theta)\) and \(\Eval{v'}{f}^{(\alpha)} = \Eval{v'}{f^{(\alpha)}}\) for all \(\alpha \in \N^d\) (see also \cref{thm:weak-smooth} below).
    By \eqref{equ:bipolar-theorem}, we obtain that for all \(h > 0\) and \(K \subseteq \RR\) compact
    \[
        p_{K,\Weight,h}(f)
        = \sup_{v' \in V_{p}^{\circ}} \norm{\Eval{v'}{f}}_{K,\Weight,h}
        , \qquad f \in \Smooth(\Theta;\Vecs)
        .
    \]
    The set \(\{ \Eval{v'}{f} \mid f \in B , v' \in V_p^\circ \}\) is bounded in \(\RDiff(\Theta)\).
    Hence, by the first part of the proof, we obtain that for all \(K \subseteq \Theta\) compact there is \(h >0\) such that 
    \[
        \sup_{f \in B} p_{K,\Weight,h}(f)
        = \sup_{f \in B} \sup_{v' \in V^\circ_p, } \norm{\Eval{v'}{f}}_{K,\Weight,h} 
        < \infty
        .
        \qedhere
    \]
\end{proof}

\begin{example}\label{exa:gevrey-space}
    Let \((\Vecs,p)\) be a normed space.
    By \cref{exa:gevrey-conjugate} and \cref{thm:inductive-description-roumieu}, \(f \in \Smooth(\Theta;\Vecs)\) belongs to \(\RDiff[\Weight_s](\Theta;\Vecs)\) if and only if for all \(K \subseteq \Theta\) compact there is \(h >0\) such that
    \[
        p_{K,s,h}(f) 
        = \sup_{x \in K} \sup_{\alpha \in \N^\Dim} \frac{p(f^{(\alpha)}(x))}{h^{\abs{\alpha}} (\abs{\alpha}!)^s} 
        < \infty
        .
    \]
    In particular, \(\RDiff[\Weight_1](\Theta;\Vecs)\) is equal to the space \(\mathcal{A}(\Theta;\Vecs)\) of \(\Vecs\)-valued real analytic functions on \(\Theta\).
\end{example}

Note that, by condition \((\beta)\), we have \(\mathcal{A}(\Theta) \subseteq \RDiff(\Theta)\).
We will always tacitly assume that \(\Weight\) satisfies \((\beta_0)\) in the Beurling case.
This assumption implies that \(\mathcal{A}(\Theta) \subseteq \BDiff(\Theta)\).
With this convention, ultradifferentiability is preserved under composition from the right with real analytic maps:

\begin{lemma}\label{thm:analytic-composition}
    Let \(\Theta, \Theta' \subseteq \RR\) be open, and let \(\Vecs\) be a lcHs.
    Let \(\phi \colon \Theta \to \Theta'\) be real analytic.
    Then, \(f \circ \phi \in \UDiff(\Theta;\Vecs)\) for all \(f \in \UDiff(\Theta';\Vecs)\).
\end{lemma}
\begin{proof} 
    This can be shown by adapting the proof of \cite[Proposition 8.4.1]{Hoermander90}; the details are left to the reader.
\end{proof}

\subsection{Ultradifferentiability on manifolds}
\label{sec:ultradifferentiability-manifold}

Fix a real analytic manifold \(\Mnf\) of dimension \(\Dim\) and let \(\Vecs\) be a lcHs.
We define \(\Smooth(\Mnf;\Vecs)\) (\(\UDiff(\Mnf;\Vecs)\)) as the space consisting of all \(f \colon \Mnf \to \Vecs\) such that \(f \circ \phi^{-1} \in \Smooth(\phi(U); \Vecs)\) (\(f \circ \phi^{-1} \in \UDiff(\phi(U); \Vecs)\)) for all real analytic charts \((\phi,U)\) of \(\Mnf\).
We endow \(\Smooth(\Mnf;\Vecs)\) (\(\UDiff(\Mnf;\Vecs)\)) with the initial topology with respect to the mappings
\begin{gather*}
    \Smooth(\Mnf;\Vecs) \to \Smooth(\phi(U); \Vecs), \, f \mapsto f \circ \phi^{-1}, \\
    (\UDiff(\Mnf;\Vecs) \to \UDiff(\phi(U); \Vecs), \, f \mapsto f \circ \phi^{-1}),
\end{gather*}
where \((\phi,U)\) runs over all (real analytic) charts of \(\Mnf\). The spaces \(\Smooth(\Mnf;\Vecs)\) and \(\UDiff(\Mnf;\Vecs)\) are sequentially complete if \(\Vecs\) is so.
We write \(\Smooth(\Mnf) = \Smooth(\Mnf;\C)\) and \(\UDiff(\Mnf) = \UDiff(\Mnf;\C)\).
Instead of a maximal atlas, one may also work with any covering atlas in the previous definitions.
This is obviously true for \(\Smooth(\Mnf;\Vecs)\) and follows from \cref{thm:analytic-composition} in the ultradifferentiable case.

Given \(v' \in \Vecs'\) and \(f \colon M \to \Vecs\), we define the mapping
\[
    \Eval{v'}{f} \colon  M \to \C, \, x \mapsto \Eval{v'}{f(x)}
    .
\]
If \(f \in \Smooth(M;\Vecs)\), then \(\Eval{v'}{f} \in \Smooth(M)\) and \(P\Eval{v'}{f} = \Eval{v'}{Pf}\) holds for each smooth linear differential operator \(P\) on \(M\).
\begin{lemma}\label{thm:weak-smooth}
    Let \(\Vecs\) be a sequentially complete lcHs and let \(f \colon \Mnf \to \Vecs\).
    \begin{enumerate}
        \item \(f \in \Smooth(\Mnf;\Vecs)\) if and only if \(\Eval{v'}{f} \in \Smooth(\Mnf)\) for all \(v' \in \Vecs'\).
        \item \(f \in \UDiff(\Mnf;\Vecs)\) if and only if \(\Eval{v'}{f} \in \UDiff(\Mnf)\) for all \(v' \in \Vecs'\).
    \end{enumerate}
\end{lemma}
\begin{proof}
    By using local coordinates, we may assume that \(M \)  is an open subset of $\R^n$.
    In this case, the first result is well-known (cf.\ \cite[Theorem 2.14]{KM}) and the second one can be derived from it in the same way as in the proof of \cite[Theorem 3.10]{Komatsu3}.
\end{proof}

Let \(\Bnds\) denote the set of non-empty absolutely convex closed bounded subsets of \(\Vecs\).
For \(B \in \Bnds\), we define \(\Vecs_B\) as the subspace of \(\Vecs\) spanned by \(B\) and endow it with the topology generated by the Minkowski functional associated with \(B\).
Since \(\Vecs\) is Hausdorff, \(\Vecs_B\) is normed.
We denote the norm by \(\norm{\Cdot}^{B}\). 
Observe that \(\Vecs_B\) is a Banach space if \(\Vecs\) is sequentially complete \cite[Corollary 23.14]{MeiseVogtBook}.

A function \(f \in \Smooth(\Mnf;\Vecs)\) is called \emph{bornologically smooth} (\emph{bornologically ultradifferentiable} of class \([\Weight]\)) if \(f \in \Smooth(\Mnf;\Vecs_B)\) (\(f \in \UDiff(\Mnf;\Vecs_B)\)) for some \(B \in \Bnds\) \cite{Komatsu3,DPV21}.
The space consisting of such functions is denoted by \(\BSmooth(\Mnf;\Vecs)\) and \(\BUDiff(\Mnf;\Vecs)\), respectively.
We endow these spaces with a convex vector bornology in the following way: a set \(\tilde{B} \subset \BSmooth(\Mnf;\Vecs)\) (\(\tilde{B} \subset \BUDiff(\Mnf;\Vecs)\)) is defined to be bounded if there exists some \(B \in \Bnds\) such that \(\tilde{B}\) is contained and bounded in \(\Smooth(\Mnf;\Vecs_B)\) (\(\UDiff(\Mnf;\Vecs_B)\)).
The spaces \(\BSmooth(\Mnf;\Vecs)\) and \(\BUDiff(\Mnf;\Vecs)\) will play an essential role in the rest of this article.

\subsection{The theorem of iterates}\label{ssec:theorem-of-iterates} 
We fix again a real analytic manifold \(\Mnf\) of dimension \(\Dim\).
Let \(P\) be an elliptic smooth linear differential operator \(P\) on \(\Mnf\).
We write \(\nu \in \N\) for the degree of \(P\).
Let \(\Vecs\) be a lcHs. 
Given a compact subset \(K \subseteq \Mnf\) and \( p \in \csn(\Vecs)\), we define for 
$j \in \N$
\[
    p_{P,K,j}(f) 
    = \sup_{i \leq j}\sup_{x \in K} p(P^if(x))
    , \qquad f \in \Smooth(\Mnf;\Vecs),
\]
and for $h >0$
$$
p_{P,K,\Weight,h}(f) 
= \sup_{j \in \N}\sup_{x \in K} p(P^jf(x))\exp \left(-\frac{1}{h}\WeightYoung(h\nu j)\right) 
, \qquad f \in \Smooth(\Mnf;\Vecs).
$$
We set \(p_{P,K,j} = \norm{\Cdot}_{P,K,j}\) and \(p_{P,K,\Weight,h} = \norm{\Cdot}_{P,K,\Weight,h}\) whenever \(\Vecs = \C\) and \(p = \abs{\Cdot}\). We define \(\BDiff_P(\Mnf;\Vecs)\) as the space consisting of all \(f \in \Smooth(\Mnf;\Vecs)\) such that $p_{P,K,\Weight,h}(f) < \infty$ for all \(K \subseteq \Mnf\) compact, $h >0$, and \(p \in \csn(\Vecs)\), and endow it 
with the Hausdorff locally convex topology generated by the system of seminorms \(\{ p_{P,K,\Weight,h} \mid K \subseteq \Mnf \text{ compact}, h >0, p \in \csn(\Vecs) \}\).
We then define
$$
\RDiff_P(\Mnf;\Vecs)= \varprojlim_{\Weight* = o(\Weight)} \BDiff[\Weight*]_P(\Mnf;\Vecs)
$$
as lcHs. 
We  write \(\UDiff_P(\Mnf) =\UDiff_P(\Mnf;\C)\).
Recall that, given \(p \in \csn(\Vecs)\), we use the notation \(V_p = \{ v \in \Vecs \mid p(v) \leq 1 \}\) and write \(V^\circ_p\) for its polar set in \(\Vecs'\).

\begin{lemma}\label{thm:elliptic-smooth}
    Let \(P\) be an elliptic smooth linear differential operator \(P\) on \(\Mnf\) and let \(\Vecs\) be a lcHs.
    The topology of \(\Smooth(\Mnf;\Vecs)\) is generated by the system of seminorms \(\{ p_{P,K,j} \mid K \subseteq \Mnf \text{ compact}, j \in \N, p \in \csn(\Vecs) \}\).
\end{lemma}
\begin{proof}
   Using local coordinates, we may assume that \(\Mnf= \Theta\) is an open subset of \(\RR\).
    Given \(j \in \N\), \(K \subseteq \Theta\) compact, and \(p \in \csn(\Vecs)\), we define, for \(f \in \Smooth(\Theta;\Vecs)\),
    \[
        p_{K,j}(f) = \sup_{|\alpha| \leq j}\sup_{x \in K} p(f^{(\alpha)}(x)).
    \]
    The topology of \(\Smooth(\Theta;\Vecs)\) is generated by the system of seminorms \(\{ p_{K,j} \mid K \subseteq \Theta \text{ compact}, j \in \N, p \in \csn(\Vecs) \}\).
    For \(\Vecs = \C\) and \(p = \abs{\Cdot}\), we set \(p_{K,j} = \norm{\Cdot}_{K,j}\). 
    The equality \eqref{equ:bipolar-theorem} yields, for all \(f \in \Smooth(\Theta;\Vecs)\),
    \[
        p_{K,j}(f) = \sup_{v' \in V^\circ_p} \norm{\Eval{v'}{f}}_{K,j}
        \qquad \text{and} \qquad
        p_{P,K,j}(f) = \sup_{v' \in V^\circ_p} \norm{\Eval{v'}{f}}_{P,K,j}
        .
    \]
    These equalities imply that it suffices to consider the case \(\Vecs = \C\). 
    We define $X$ as the space consisting of all distributions $f \in \mathcal{D}'(\Theta)$ such that $P^j f \in C(\Theta)$ for all $j \in \N$, where $P^jf$ should be understood in the sense of distributions. We endow $X$ with the Hausdorff locally convex topology generated by the system of seminorms \(\{ \norm{\Cdot}_{K,j} \mid K \subseteq \Theta \text{ compact}, j \in \N \}\). Note that $X$ is a Fr\'echet space. It suffices to prove that $C^\infty(\Theta) = X$ as lcHs. By the open mapping theorem, it is enough to show that 
    $C^\infty(\Theta) \subseteq X$ with continuous inclusion and $X \subseteq C^\infty(\Theta)$ as sets.  The former is clear, we now show the latter.     Let $H^{\operatorname{loc}}_s(\Theta)$ be the local Sobolev space of order $s \in \R$.
    The elliptic regularity theorem for local Sobolev spaces  \cite[p.\ 214, (6.33)]{FollandPDE} implies that for all $s \in \R$
    \begin{equation}
        Pf \in H^{\operatorname{loc}}_s(\Theta) \, \Longrightarrow f \in  H^{\operatorname{loc}}_{s+ \nu}(\Theta), \qquad \forall f \in \mathcal{D}'(\Theta).
        \label{elliptic-reg}
    \end{equation}
    Let $f \in X$ be arbitrary. Note that $P^j f \in C(\Theta) \subseteq L^{2, \operatorname{loc}}(\Theta) =   H^{\operatorname{loc}}_0(\Theta)$ for all $j \in \N$. Hence, \eqref{elliptic-reg} implies that
    \[
        f \in \bigcap_{j \in \N} H^{\operatorname{loc}}_{\nu j}(\Theta) = C^\infty(\Theta).
        \qedhere
    \]
\end{proof}

The following result is essentially shown in \cite{FS}. 
\begin{theorem}
    \label{thm:elliptic-ultradifferentiable}
    Let \(P\) be an elliptic analytic linear differential operator \(P\) on \(\Mnf\) and let \(\Vecs\) be a lcHs.
    Then, \(\UDiff(\Mnf;\Vecs) = \UDiff_P(\Mnf;\Vecs)\) as lcHs.
\end{theorem}
\begin{proof}
    By definition of the spaces in the Roumieu case, it is enough to consider the Beurling case.
    Using local coordinates, we may assume that \(\Mnf= \Theta\) is an open subset of \(\RR\).
    The equality \eqref{equ:bipolar-theorem} yields, for all \(p \in \csn(\Vecs)\), \(h >0\), \(K \subseteq \Theta\) compact, and \(f \in \Smooth(\Theta;\Vecs)\),
    \[
        p_{ K,\Weight,h}(f) = \sup_{v' \in V^\circ_p} \norm{\Eval{v'}{f}}_{K,\Weight,h}
        \qquad \text{and} \qquad
        p_{P, K,\Weight,h}(f) = \sup_{v' \in V^\circ_p} \norm{\Eval{v'}{f}}_{P, K,\Weight,h}        .
    \]
    So, these equalities imply that, once again, it suffices to consider the case \(\Vecs = \C\).
    We start by showing that \(\BDiff(\Theta) = \BDiff_P(\Theta)\) as sets.
    It follows from the proof of \cite[Proposition 3.2 3]{FS} (applied to the weight matrix associated with \(\Weight\), see \cite[Definition 2.9]{FS}) that \(\BDiff(\Theta) \subseteq \BDiff_P(\Theta)\).
    Next, we define \(\BDiff_{P,2}(\Theta)\) as the space consisting of all \(f \in \Smooth(\Theta)\) such that, for all \(h >0\) and \(K \subseteq \Theta\) compact,
    \begin{equation*}
        \sup_{j \in \N} \norm{P^j f}_{L^2(K)}\exp \left(-\frac{1}{h}\WeightYoung(h\nu j)\right) < \infty   .
    \end{equation*}
    It is clear that \(\BDiff_P(\Theta) \subseteq \BDiff_{P,2}(\Theta)\).
    Moreover, it is shown in \cite[Corollary 3.19]{FS} that \(\BDiff_{P,2}(\Theta) = \BDiff(\Theta)\).
    Therefore, we conclude that 
    \[
        \BDiff_P(\Theta) \subseteq \BDiff_{P,2}(\Theta) 
        =  \BDiff(\Theta)
        ,
    \]
    which establishes the claimed equality \(\BDiff(\Theta) = \BDiff_P(\Theta)\).
    We still need to verify that \(\BDiff(\Theta) = \BDiff_P(\Theta)\) as lcHs.
    \Cref{thm:elliptic-smooth} yields that \(\BDiff_P(\Theta)\) is continuously embedded into \(\Smooth(\Theta)\).
    This implies that \(\BDiff_P(\Theta)\) is a Fréchet space.
    Note that \(\BDiff(\Theta)\) is also a Fréchet space.
    Hence, by the closed graph theorem, it suffices to show that the inclusion mappings \(\BDiff(\Theta) \to \BDiff_P(\Theta)\) and \(\BDiff_P(\Theta) \to \BDiff(\Theta)\) have closed graph.
    This follows from the fact that \(\BDiff(\Theta)\) and \(\BDiff_P(\Theta)\) are continuously embedded into \(\Smooth(\Theta)\).
\end{proof}

\section{Ultradifferentiable vectors and the factorization theorem}
\label{sec:main-result}
Throughout the rest of this article we fix a compact (real) Lie group \(\LieG\) of dimension \(\Dim\) with identity element \(\Origin\).
We view the underlying manifold structure of \(\LieG\) as a real analytic one.
Furthermore, \(\Vecs\) will always denote a sequentially complete lcHs.
We fix a normalized Haar measure on \(\LieG\).
All integrals on \(\LieG\) will be meant with respect to this measure and all vector-valued integrals are to be interpreted in the weak sense (Gelfand-Pettis integral).
Since \(\LieG\) is compact and \(\Vecs\) is sequentially complete, the integral \(\int_\LieG f(x) \dl{x}\) exists for all \(f \in \Cont(\LieG;\Vecs)\).

We denote the group of linear topological isomorphisms of \(\Vecs\) by \(\GGL{\Vecs}\).
A \emph{representation} of \(\LieG\) on \(\Vecs\) is merely a group homomorphism \(\Rep \colon \LieG \to \GGL{\Vecs}\).
For \(v \in \Vecs\), the mapping
\[
    \Orbit{v} \colon \LieG \to \Vecs
    , \quad x \mapsto \Rep(x)v
    ,
\]
is called the \emph{orbit} of \(v\).
We denote by \(\SmoothV[0]\) the space consisting of all \(v \in \Vecs\) such that \(\Orbit{v} \in \Cont(\LieG;\Vecs)\).
The representation is said to be \emph{continuous} if \(\Vecs= \SmoothV[0]\).

From now on, let \(\Rep\) be a representation of \(\LieG\) on \(\Vecs\).
We define the spaces of smooth and ultradifferentiable vectors of class \([\Weight]\) as
\[
    \BSmoothV = \{v \in \Vecs \mid \Orbit{v} \in \BSmooth(\LieG;\Vecs)\} 
    \qquad \text{and} \qquad
    \BUDiffV = \{v \in \Vecs \mid \Orbit{v} \in \BUDiff(\LieG;\Vecs)\},
\]
respectively.
These spaces inherit a convex vector bornology from their natural embedding into \(\BSmooth(\LieG;\Vecs)\) and \(\BUDiff(\LieG;\Vecs)\).
Namely, a subset of vectors is bounded in \(\SmoothV\) (\(\BUDiffV\)) if the associated set of orbits is bounded in \(\BSmooth(\LieG;\Vecs)\) (\(\BUDiff(\LieG;\Vecs)\)).
We refer to \cite{D-H-V} for the study of some linear topological properties of (Beurling type) spaces of ultradifferentiable vectors.

\begin{remark}\label{rem:analytic-vectors-coincide}
    If \(\Vecs\) is a Fréchet space, the space \(\RAnalV\) of analytic vectors considered in \cite{G-K-L} coincides with our space \(\RDiffV[\Weight_1]\).
    By going to local coordinates, this follows from \cref{exa:gevrey-space} and \cite[Proposition 12]{B-D} together with the  remark \cite[`'Note added in Proof'', p.\ 34]{B-D}.
\end{remark}

We recall that $C(G) = C(G;\C)$ and that the same convention is used for other spaces as well.
The convolution of \(\Molf \in \Cont(\LieG)\) and \(f \in \Cont(\LieG;\Vecs)\) is defined as
\[
    (\Molf \ast f)(x) 
    = \int_{\LieG} \Molf(y) f(y^{-1}x) \dl{y}
    = \int_{\LieG} \Molf(xy) f(y^{-1}) \dl{y}
    , \qquad x \in \LieG
    .
\]
We have,
\begin{align}
    \label{equ:smoothening-smooth}
    \Smooth(\LieG) \ast \Cont(\LieG;\Vecs) &\subseteq \Smooth(\LieG;\Vecs), \\
    \label{equ:smoothening-ultradifferentiable}
    \UDiff(\LieG) \ast \Cont(\LieG;\Vecs) &\subseteq \UDiff(\LieG;\Vecs)
    .
\end{align}

As in the Introduction, we define,
for \(v \in \SmoothV[0]\) and \(\Molf \in \Cont(\LieG)\), 
\[
    \Action(\Molf)v 
    = \int_{\LieG} \Molf(y) \Orbit{v}(y) \dl{y} 
    = \int_{\LieG} \Molf(y) \, \Rep(y)v \dl{y} 
    .
\]
In particular, we have, for \(\Molf_{1},\Molf_{2} \in \Cont(\LieG)\), 
\begin{equation}
    \label{equ:algebra-representation}
    \Action(\Molf_1 \ast \Molf_{2})
    = \Action(\Molf_1) \circ \Action(\Molf_{2}).
\end{equation}

\begin{example}
    \label{exa:regular-representation}
    Consider the left regular representation \(\LRep\) on \(\Cont(\LieG;\Vecs)\), i.e., \(\LRep(x)f = L_x f\) for \(f \in \Cont(\LieG;\Vecs)\) and \(x \in \LieG\), where, as usual, \(L_y f(x) = f(y^{-1}x)\).
    This representation is always continuous.
    One can show that the spaces of smooth and ultradifferentiable vectors associated with this representation are given by \(\BSmooth(\LieG;\Vecs)\) and \(\BUDiff(\LieG;\Vecs)\), respectively.
    Clearly, \(\Action_L(\Molf)f = \Molf \ast f\) for \(\Molf \in \Cont(\LieG)\) and \(f \in \Cont(\LieG;\Vecs)\), where \(\Action_L\) denotes the action induced by \(\LRep\). 
\end{example}

As \(\Action\) is based on convolution, it smoothens vectors:
\begin{lemma}
    \label{thm:smoothening}
    Let \(\Rep\) be a representation of \(\LieG\) on \(\Vecs\), then
    \[
        \Action(\Smooth(\LieG))\SmoothV \subseteq \SmoothV 
        \qquad \text{and} \qquad
        \Action(\UDiff(\LieG))\BUDiffV \subseteq \BUDiffV
        .
    \]
\end{lemma}
\begin{proof}
    Note that
    \[
        \check{\gamma}_{\Action(\Molf)v }(x)
        = \Rep(x^{-1}) \Action(\Molf) v
        = \int_{\LieG} \Molf(y) \, \Rep(x^{-1}y)v \dl{y} 
        = \Molf \ast \check{\gamma}_{v}(x),
    \] 
    where \(\check{\gamma}_{v}(x) = \Orbit{v}(x^{-1})\). 
    Hence, the result follows from \eqref{equ:smoothening-smooth} and \eqref{equ:smoothening-ultradifferentiable}.
\end{proof}

The space \(\SmoothV\) (\(\BUDiffV\)) is then a left module over \(\Smooth(\LieG)\) (\(\UDiff(\LieG)\)) in view of \eqref{equ:algebra-representation} and \cref{thm:smoothening}. 
We are now ready to state the main result of this article.

\begin{theorem}
    \label{thm:factorization-vectors}
    Let \(\Rep\) be a representation of \(\LieG\) acting on \(\Vecs\).
    Then, \(\BUDiffV\) has the bounded strong factorization property as a left module over \(\UDiff(\LieG)\). 
    In particular,
    \[
        \BUDiffV
        = \Action(\UDiff(\LieG)) \BUDiffV
        .
    \]
\end{theorem}

\begin{remark} \cref{th analytic} follows at once from \cref{thm:factorization-vectors} in the Roumieu case with \(\Weight = \Weight_1\) and \(\Vecs\) a Fréchet space (cf.\ \cref{rem:analytic-vectors-coincide}).
\end{remark}

We postpone the proof of \cref{thm:factorization-vectors} to \cref{sec:main-proof}, but we shall now discuss some of its corollaries.
Firstly, we find the following result (see \cref{exa:regular-representation}) by considering the left regular representation \(\LRep\) on \(\Cont(\LieG;\Vecs)\).
\begin{corollary}
    \label{thm:factorization-functions}
    The space \(\BUDiff(\LieG;\Vecs)\) has the bounded strong factorization property as a left module over \(\UDiff(\LieG)\).
    In particular,
    \[
        \BUDiff(\LieG;\Vecs)
        = \UDiff(\LieG) \ast \BUDiff(\LieG;\Vecs)
        .
    \]
\end{corollary}

\Cref{thm:factorization-vectors} can also be used to establish the strong factorization property for spaces \(\SmoothV\) of smooth vectors.
In the particular case of a Fréchet space \(\Vecs\), one recovers the strong factorization result for smooth vectors on compact Lie groups of Dixmier and Malliavin \cite[Théorème 4.11]{D-M}.
Notice our proof is completely different from theirs.

\begin{corollary}\label{thm:factorization-vectors-smooth}
    Let \(\Rep\) be a representation of \(\LieG\) acting on \(\Vecs\).
    Then, \(\SmoothV\) has the bounded strong factorization property as a left module over \(C^{\infty}(G)\). 
    In particular,
    \[
        \SmoothV
        = \Action(\Smooth(\LieG)) \SmoothV
        .
    \]
\end{corollary}
\begin{proof} 
    This follows from \cref{thm:factorization-vectors} and \cref{thm:smooth-to-roumieu} (shown below).
\end{proof}

Specializing \cref{thm:factorization-vectors-smooth} to the left regular representation \(\LRep\) on \(\Cont(\LieG;\Vecs)\), we obtain (see \cref{exa:regular-representation}):

\begin{corollary}\label{thm:factorization-functions-smooth}
    The space \(\BSmooth(\LieG;\Vecs)\) has the bounded strong factorization property as a left module over \(C^{\infty}(\LieG)\).
    In particular,
    \[
        \BSmooth(\LieG;\Vecs)
        = \Smooth(\LieG) \ast \BSmooth(\LieG;\Vecs)
        .
    \]
\end{corollary}

\begin{remark}
    Suppose that \(\Vecs\) is a Fréchet space.
    In \cite[Théorèmes 4.9 and 4.11]{D-M}, Dixmier and Malliavin showed \cref{thm:factorization-vectors-smooth,thm:factorization-functions-smooth} but with \(\tilde{E}^{\infty} = \{v \in \Vecs \mid \Orbit{v} \in \Smooth(\LieG;\Vecs)\}\) and \(\Smooth(\LieG;\Vecs)\) instead of \(\SmoothV\) and \(\Smooth_b(\LieG;\Vecs)\), respectively.
    However, since \(\Vecs\) is a Fréchet space, one has \(\Smooth(\LieG;\Vecs) = \BSmooth(\LieG;\Vecs)\) (cf.\ \cite[Lemma 2.7]{JH}) and thus also \(\tilde{E}^{\infty} = \SmoothV\).
\end{remark}

\begin{remark}
    One might wonder whether the spaces
    \[
        \tilde{E}^{\infty} = \{v \in \Vecs \mid \Orbit{v} \in \Smooth(\LieG;\Vecs)\} 
        \qquad \text{and} \qquad
        \tilde{E}^{[\Weight]} = \{v \in \Vecs \mid \Orbit{v} \in \UDiff(\LieG;\Vecs)\}
    \]
    have the strong factorization property for a general \(\Vecs\).
    By applying the mean value theorem in local coordinates, it follows that for any \(f \in \Smooth(\LieG;\Vecs)\) there exists some \(B \in \Bnds\) such that \(f \in \Cont(\LieG;\Vecs_B)\) (cf.\ \cite[Lemma 1]{BJM}).
    By \eqref{equ:smoothening-smooth} and \eqref{equ:smoothening-ultradifferentiable}, this yields that
    \[
        \Action(\Smooth(\LieG))\tilde{E}^{\infty} \subseteq \BSmoothV \subseteq \tilde{E}^{\infty} 
        \qquad \text{and} \qquad
        \Action(\UDiff(\LieG))\tilde{E}^{[\Weight]} \subseteq \BUDiffV \subseteq \tilde{E}^{[\Weight]}
        .
    \]
    Hence, in general, one can only hope for strong factorization for \(\SmoothV\) \((\UDiffV)\).
    Moreover, if \(\tilde{E}^{\infty}\) (\(\tilde{E}^{[\Weight]}\)) has the strong factorization property, then we must necessarily have \(\BSmoothV = \tilde{E}^{\infty}\) (\(\BUDiffV = \tilde{E}^{[\Weight]}\)).
\end{remark}

\section{Ultradifferentiability via the Laplace-Beltrami operator}
\label{sec:laplacian}
Let \(\Lapl\) be the Laplace-Beltrami operator on \(\LieG\) (with respect to a fixed bi-invariant Riemannian metric on \(\LieG\)).
Then, \(\Lapl\) is a second-order elliptic analytic linear differential operator on \(\LieG\) that is bi-invariant, negative semi-definite and formally self-adjoint (see e.g.\ \cite[p.~35]{Stein70}).
The spectrum of \(-\Lapl\) solely consists of an unbounded sequence of eigenvalues (see e.g.\ \cite{TrevesBook1980II}).

We now use the results from \cref{ssec:theorem-of-iterates} to describe the spaces \(\Smooth(\LieG;\Vecs)\) and \(\UDiff(\LieG;\Vecs)\) in terms of the Laplace-Beltrami operator \(\Lapl\).
Let \(f \in \Smooth(\LieG;\Vecs)\), \(p \in \csn(\Vecs)\), \(j \in \N\), and \(h>0\).
To ease our notation, we write \(\psupnorm(f) = \sup_{x \in \LieG} p(f(x))\) and abbreviate the norms used in \cref{ssec:theorem-of-iterates} by
\begin{alignat*}{3}
    p_{j}(f) &\coloneqq p_{\LaplP*,G,j}(f) &&= \max_{i \leq j} \psupnorm(\LaplP^{i}f), \\
    p_{\Weight,h}(f) &\coloneqq p_{\LaplP*,G,\Weight,h}(f) &&= \sup_{j \in \N} \psupnorm(\LaplP^{j}f)\exp \left(-\frac{1}{h}\WeightYoung(2jh)\right) 
    ,
\end{alignat*}
where the factor $2$ in the defintion of $ p_{\Weight,h}$ stems from the fact that the degree of the Laplacian is equal to 2.
\Cref{thm:elliptic-smooth,thm:elliptic-ultradifferentiable} yield the following results:
\begin{proposition}
    \label{thm:smoothness-laplacian}
    Let \(\LieG\) be a compact Lie group and let \(\Vecs\) be a lcHs.
    The system of seminorms \(\{p_{j} \mid p \in \csn(\Vecs), j \in \N\}\) induces the topology on \(\Smooth(\LieG;\Vecs)\).
\end{proposition}

\begin{proposition}
    \label{thm:ultradifferentiability-laplacian}
    Let \(\LieG\) be a compact Lie group and let \(\Vecs\) be a lcHs. 
    Then $\UDiff(\LieG;\Vecs) =   \UDiff[\Weight]_{\LaplP*}(\LieG;\Vecs)$ as lcHs. 
\end{proposition}

In connection with Proposition \ref{thm:smoothness-laplacian}, the reader can also consult \cite{G-K}, where Sobolev norms for general Lie groups are studied with respect to Laplace elements of the universal enveloping algebra.

\section{Characterizing ultradifferentiable functions via their Fourier coefficients}
\label{sec:fourier-coefficients}
In this section, we will use the Peter-Weyl Theorem to characterize the spaces \(\UDiff(\LieG;\Vecs)\) in terms of the group Fourier transform.
This gives a discrete description of \(\UDiff(\LieG;\Vecs)\) and will be the main tool in the proof of our main result \cref{thm:factorization-vectors}, given in the next section.

We start by recalling the Peter-Weyl Theorem, see \cite[Sections 5.2 and 5.3]{Folland95} for more information.
Let \(\IrredRep \colon \LieG \to U(\IrredRepS)\) be a continuous irreducible unitary representation of \(\LieG\), where \(U(\IrredRepS)\) is the space of unitary operators on a  (complex) Hilbert space \(\IrredRepS\).
Since \(\LieG\) is compact, \(\IrredRepS\) is finite dimensional and we write \(\IrredRepD\) for its dimension.
Furthermore, as \(U(\IrredRepS)\) is a Lie group, the continuity of the map \(\IrredRep \colon \LieG \to U(\IrredRepS)\) implies that it is in fact smooth. 

Two irreducible unitary representations \(\IrredRep_1 \colon \LieG \to U( {\mathcal{H}}_{\xi_1}) \) 
and  \(\IrredRep_2 \colon \LieG \to U( {\mathcal{H}}_{\xi_2})\) 
of $G$ are called unitarily equivalent if there exists a unitary intertwining  operator for $\xi_1$ and $\xi_2$, i.e., a unitary operator $T: {\mathcal{H}}_{\xi_1} \to {\mathcal{H}}_{\xi_2} $
such that $T \circ (\xi_1(x)) = (\xi_2(x)) \circ T$ for all $x \in G$. This defines an equivalence relation on the space of all  irreducible unitary representations of $G$.
We denote by \(\DLieG\) the set of unitary equivalence classes \(\IrredRepC\) of continuous irreducible unitary representations \(\IrredRep \colon \LieG \to U(\IrredRepS)\).
The notation \(\IrredRep \in \DLieG\) is to be understood as \(\IrredRepC \in \DLieG\), where we fix a particular representative.

Pick \(\IrredRep \in \DLieG\) and fix an orthonormal basis \(\{e_{i} \mid 1 \leq i \leq \IrredRepD\}\) of \(\IrredRepS\).
We call the smooth functions \(\IrredRep_{i,j}(x) = (e_{j},\IrredRep(x)e_{i})_{\IrredRepS}\) 
\emph{matrix coefficients}.
Now, define \(\MatCoefS = \Span\{\IrredRep_{i,j} \mid 1 \leq i,j \leq \IrredRepD\}\).
The Peter-Weyl Theorem states that \(L^{2}(\LieG) = \bigoplus_{\IrredRep \in \DLieG} \MatCoefS\) and that, for all \(f \in L^{2}(\LieG)\),
\begin{equation}
    \label{equ:peter-weyl-matrix-coefficients}
    f
    = \sum_{\IrredRep \in \DLieG} \IrredRepD\sum_{i,j = 1}^{\IrredRepD} c_{i,j}^{\IrredRep}(f) \IrredRep_{i,j}
    , \qquad c_{i,j}^{\IrredRep}(f) = \int _{\LieG} f(x) \overline{\IrredRep_{i,j}(x)} \dl{x}.
\end{equation}
This series converges in \(L^{2}(\LieG)\).

The matrix coefficients associated with any \(\IrredRep \in \DLieG\) are eigenfunctions of \(\Lapl\) with a common eigenvalue (see e.g.\ \cite[Chapter 10]{Fegan91}), that is, there is \(\IrredRepE \geq 0\) (as \(\Lapl\) is negative semi-definite) such that
\begin{equation}
    \label{equ:matrix-coefficient-eigenfunction}
    \Lapl \IrredRep_{i,j} 
    = - \IrredRepE \IrredRep_{i,j} 
    ,\qquad \text{for all } i,j = 1, \ldots, \IrredRepD
    .
\end{equation}
Since the matrix coefficients form an orthogonal basis of eigenfunctions, these are all the eigenvalues of \(\Lapl\).
Let \(N(\lambda)\) be the eigenvalue counting function of \(-\Lapl\).
Weyl's law (see e.g. \cite[p.~624]{TrevesBook1980II}) then implies that 
\begin{equation}
    \label{equ:summability-eigenvalues}
    \sum_{\IrredRep \in \DLieG} \IrredRepD^{2} \IrredRepEP^{-\alpha}= \int_{0^{-}}^{\infty} (1+\lambda)^{-\alpha} \dl N(\lambda) < \infty 
\end{equation}
if and only if \(\alpha > \frac{\Dim}{2}\), with $n = \operatorname{dim}(G)$.
In what follows, we will employ \eqref{equ:summability-eigenvalues} with \(\alpha=n\).

We define the group Fourier transform of \(f \in L^2(G)\) as
\[
    \FTO(f) (\IrredRep)
    = \FT f (\IrredRep)
    = \Xi(f)= \int _{\LieG} f(x) \IrredRep(x) \dl{x}
    \in \mathcal{L}(\IrredRepS)
    , \qquad \IrredRep \in \DLieG,
\]
where \(\Xi\) denotes the induced action by the representation \(\IrredRep\) (which is in this case well-defined even for  \(f\in L^{1}(\LieG)\)). 
If we pick an orthonormal basis \(\{e_{i} \mid 1 \leq i \leq \IrredRepD\}\) of \(\IrredRepS\), then \(\FT f (\IrredRep)\) is given by the matrix \((c_{j,i}^{\IrredRep}(f))_{i,j =1, \ldots,\IrredRepD}\).
Hence, we can rewrite \eqref{equ:peter-weyl-matrix-coefficients} as a Fourier inversion formula
\begin{equation}
    \label{equ:fourier-inversion}
    f(x) 
    = \sum_{\IrredRep \in \DLieG} \IrredRepD \Trace [\IrredRep^{\ast}(x) \circ \FT f (\IrredRep)] 
    .
\end{equation} 
Furthermore, the following Parseval equality holds
\[
    \norm{f}_{2}^{2} 
    = \sum_{\IrredRep \in \DLieG} \IrredRepD \HSnorm{\FT f (\IrredRep)}^{2}
    ,
\]
where \(\HSnorm{T} = \sqrt{\Trace(T \circ T^{*})}\) is the Hilbert-Schmidt norm of \(T \in \mathcal{L}(\IrredRepS)\).

Define
\[
    \ell_{2}(\DLieG)
    \coloneqq \{T = (T_{\IrredRep})_{\IrredRep \in \DLieG} \in \prod_{\IrredRep \in \DLieG} \mathcal{L}(\IrredRepS) \mid \norm{T}_{2}^{2} = \sum_{\IrredRep \in \DLieG} \IrredRepD \HSnorm{T_{\IrredRep}}^{2} < \infty\}
    .
\]
Then \(\FTO \colon L^{2}(\LieG) \to \ell_{2}(\DLieG)\) is an isometric isomorphism and, because of \eqref{equ:fourier-inversion}, its inverse is given by
\[
    \FTO^{-1}(T)(x) 
    = \sum_{\IrredRep \in \DLieG} \IrredRepD \Trace [\IrredRep^{\ast}(x) \circ T_{\IrredRep} ]
    , \qquad T = (T_{\IrredRep})_{\IrredRep \in \DLieG} \in \ell_{2}(\DLieG)
    .
\]

Let us now adapt the group Fourier transform to the vector-valued setting.
Fix a sequentially complete lcHs \(E\).
For \(f \in \Cont(\LieG;\Vecs)\), we define
\[
    \FTO(f)(\IrredRep)
    = \FT f (\IrredRep)
    \coloneqq \int _{\LieG} f(x) \otimes \IrredRep(x) \dl{x}
    \in \Vecs \otimes \mathcal{L}(\IrredRepS) 
    , \qquad \IrredRep \in \DLieG
    .
\]
Note that
\[
    \FT{ \Eval{v'}{f }}(\IrredRep)
    = \Eval{v'}{\FT f (\IrredRep)}
    , \qquad v' \in \Vecs'
    ,
\]
where we simply write \(\Eval{v'}{\cdot}\) for the map \(\Eval{v'}{\cdot} \otimes \Id_{\IrredRepS}\) on the right-hand side.
For \(p \in \csn(\Vecs)\), we define
\[
    \pHSnorm(T) 
    \coloneqq \sup_{v' \in V_{p}^{\circ}} \HSnorm{\Eval{v'}{T}}
    , \qquad T \in \Vecs \otimes \mathcal{L}(\IrredRepS)
    .
\]
We endow \(\Vecs \otimes \mathcal{L}(\IrredRepS)\) with the Hausdorff locally convex topology generated by the system 
of seminorms \( \{\pHSnorm \mid p \in \csn(\Vecs)\}\).

We now show two standard expected properties of the group Fourier transform in the vector-valued setting.
\begin{lemma}
    \label{thm:eigenvalues-laplacian}
    For all \(f \in \Smooth(\LieG;\Vecs)\),
    \[
        \FT{\Lapl f}(\IrredRep)
        = -\IrredRepE\FT{f}(\IrredRep)
        \qquad \IrredRep \in \DLieG
        .
    \]
\end{lemma}
\begin{proof}
    In the scalar-valued case this follows from \eqref{equ:matrix-coefficient-eigenfunction} and the fact that \(\Lapl\) is formally self-adjoint.
    The general case then holds, since, for any \(v' \in \Vecs'\), we have
    \begin{align*}
        \Eval{v'}{\FTO (\Lapl f)(\IrredRep)}
         &= \FTO\left(\Eval{v'}{\Lapl f}\right)(\IrredRep)
         = \FTO\left(\Lapl \Eval{v'}{f}\right)(\IrredRep)
         = -\IrredRepE \FTO \left(\Eval{v'}{f}\right)(\IrredRep)\\
         &= \Eval{v'}{-\IrredRepE \FTO(f)(\IrredRep)}
         .
         \qedhere
    \end{align*}
\end{proof}

We define the composition \(A \circ B \in \Vecs \otimes \mathcal{L}(\IrredRepS)\) for \(A \in \mathcal{L}(\IrredRepS)\) and \(B \in \Vecs \otimes \mathcal{L}(\IrredRepS)\) via \((\Id_{\Vecs} \otimes (A \circ \cdot))(B)\).
\begin{lemma}
    \label{thm:convolution-to-composition}
    For all \(f \in \Cont(\LieG)\) and \(g \in \Cont(\LieG;\Vecs)\), 
    \[
        \FT {f \ast g} (\IrredRep)
        =  \FT{f} (\IrredRep) \circ \FT{g} (\IrredRep)
        \qquad \IrredRep \in \DLieG
        .
    \]
\end{lemma}
\begin{proof}
    We have 
    \begin{align*}
        \FT {f \ast g} (\IrredRep)
         &= \int_G \left(\int_G f(y) g(y^{-1}x) \dl{y} \right) \otimes \IrredRep(x) \dl{x} \\
         &= \int_G \left(\int_G g(y^{-1}x) \otimes f(y) \IrredRep(x) \dl{x}\right) \dl{y} \\
         &= \int_G g(x) \otimes \left( \left(\int_{G} f(y) \IrredRep(y) \dl{y}\right)\circ \IrredRep(x)\right) \dl{x} \\
         &= \left(\int_G f(y) \IrredRep(y) \dl{y}\right) \circ \left(\int_G g(x) \otimes \IrredRep(x) \dl{x} \right)\\
         &=\FT{f} (\IrredRep) \circ \FT{g} (\IrredRep)
         .
         \qedhere
    \end{align*}
\end{proof}

We define \(\SmoothC(\DLieG;\Vecs)\) as the space consisting of all \(T = (T_{\IrredRep})_{\IrredRep} \in \prod_{\IrredRep \in \DLieG} \Vecs \otimes \mathcal{L}(\IrredRepS)\) such that, for all \(p \in \csn(\Vecs)\) and \(j \in \N\),
\[
    \FT{p}_{j}(T) 
    = \sup_{\IrredRep \in \DLieG} \pHSnorm(T_{\IrredRep}) (1+\IrredRepE)^{j} 
    < \infty
    .
\]
We endow this space with the Hausdorff locally convex topology generated by the system of seminorms \(\{ \FT{p}_{j} \mid j \in \N, p \in \csn(\Vecs) \}\).
We write \(\SmoothC(\DLieG) = \SmoothC(\DLieG;\C)\).

By \eqref{equ:summability-eigenvalues}, \(\SmoothC(\DLieG)\) is continuously embedded into \(\ell_{2}(\DLieG)\).
For \(T = (T_{\IrredRep})_{\IrredRep} \in \SmoothC(\DLieG;\Vecs)\), we define 
\[
    \FTO^{-1}(T)(x)
    = \sum_{\IrredRep \in \DLieG} \IrredRepD \Trace [\IrredRep^{\ast}(x)\circ T_{\IrredRep}]
    ,
\]
where \(\Trace\) is to be understood as \(\Id_{\Vecs} \otimes \Trace \).

\begin{proposition}
    \label{thm:characterization-coefficients-smooth}
    The Fourier transform \(\FTO \colon \Smooth(\LieG;\Vecs) \to \SmoothC(\DLieG;\Vecs)\) is an isomorphism of lcHs and \(\FTO^{-1}\) is its inverse.
\end{proposition}
\begin{proof}
    We start by showing that \(\FTO \colon \Smooth(\LieG;\Vecs) \to \SmoothC(\DLieG;\Vecs)\) is continuous. 
    Let \(p \in \csn(\Vecs)\) and \(j \in \N\) be arbitrary.
    By \eqref{equ:summability-eigenvalues}, there is \(C_1 > 0\) such that \(\IrredRepD \leq C_1 \IrredRepEP^{\Dim}\) for all \(\IrredRep \in \DLieG\).
    Hence, for all \(f \in \Smooth(\LieG;\Vecs)\)
    \begin{align*}
        \IrredRepEP^{j} \, \pHSnorm(\FT f (\IrredRep)) 
       &= \IrredRepEP^{-\Dim}\pHSnorm\left(\FTO ({\LaplP^{j+\Dim} f)} (\IrredRep)\right)\\
       &= \IrredRepEP^{-\Dim} \sup_{v' \in V_{p}^{\circ}}\HSnorm{\FTO \left(\Eval{v'}{\LaplP^{j+\Dim} f}\right) (\IrredRep)} \\
       &\leq \IrredRepEP^{-\Dim}\sup_{v' \in V_{p}^{\circ}}\int _{\LieG} \abs*{\Eval{v'}{\LaplP^{j+\Dim} f(x)}} \HSnorm{\IrredRep(x)} \dl{x} \\
       &\leq \IrredRepEP^{-\Dim} \IrredRepD \int _{\LieG} \sup_{v' \in V_{p}^{\circ}}\abs*{\Eval{v'}{\LaplP^{j+\Dim} f(x)}} \dl{x} \\
       &\leq C_1 \sup_{x \in \LieG} p\left(\LaplP^{j+\Dim} f(x)\right)
       .
    \end{align*}
    By \cref{thm:smoothness-laplacian}, the map \(\FTO \colon \Smooth(\LieG;\Vecs) \to \SmoothC(\DLieG;\Vecs)\) is continuous.

    Conversely, we now prove that \(\FTO^{-1} \colon \SmoothC(\DLieG;\Vecs) \to \Smooth(\LieG;\Vecs) \) is continuous.
    By \eqref{equ:summability-eigenvalues}, we have  \(C_2 = \left(\sum_{\IrredRep \in \DLieG} \IrredRepD^{2} \IrredRepEP^{-\Dim}\right) < \infty\).
    Let \(T = (T_{\IrredRep})_{\IrredRep} \in \SmoothC(\DLieG;\Vecs)\) be arbitrary.
    For all \(p \in \csn(\Vecs)\) and \(j \in \N\),
    \begin{align}
        \label{equ:estimate-fourier-series}
        \begin{split}
            \sum_{\IrredRep \in \DLieG} \IrredRepD p_j\left(\Trace [ \IrredRep^{\ast}(x) \circ T_{\IrredRep}  ]\right)
             &=\sum_{\IrredRep \in \DLieG} \IrredRepD \sup_{i \leq j} \sup_{x \in G} \sup_{v' \in V_{p}^{\circ}} \abs*{\Eval{v'}{\LaplP^{i}\left(\Trace [\IrredRep^{\ast}(x)\circ T_{\IrredRep} ] \right)}} \\
             &\leq \sum_{\IrredRep \in \DLieG} \IrredRepD \sup_{i \leq j} \sup_{x \in G} \sup_{v' \in V_{p}^{\circ}}\abs*{\LaplP^{i}\Trace [\IrredRep^{\ast}(x) \circ \Eval{v'}{T_{\IrredRep}}] } \\
             &\leq \sum_{\IrredRep \in \DLieG} \IrredRepD \IrredRepEP^{j} \sup_{x \in G} \sup_{v' \in V_{p}^{\circ}} \HSnorm{\Eval{v'}{T_{\IrredRep}}} \HSnorm{\IrredRep(x)} \\
             &\leq \sum_{\IrredRep \in \DLieG} \IrredRepD^{2} \IrredRepEP^{j} \pHSnorm(T_{\IrredRep}) \\
             &\leq C_2 \FT{p}_{j+\Dim}(T)
             .
        \end{split}
    \end{align}
    Here, we used the Cauchy-Schwartz inequality for the Hilbert-Schmidt norm.
    This shows that the series
    \[
        \FTO^{-1}(T) 
        = \sum_{\IrredRep \in \DLieG} \IrredRepD \Trace [\IrredRep^{\ast}\circ T_{\IrredRep} ]
    \]
    is absolutely summable and hence summable in \(\Smooth(G;E)\) (as this space is sequentially complete).
    Moreover, by \eqref{equ:estimate-fourier-series}, we have, for all \(p \in \csn(\Vecs)\) and \(j \in \N\),
    \[
        p_j( \FTO^{-1}(T)) 
        \leq C_2 \FT{p}_{j+\Dim}(T)
        .
    \]
    By \cref{thm:smoothness-laplacian}, the map \(\FTO^{-1} \colon \SmoothC(\DLieG;\Vecs) \to \Smooth(\LieG;\Vecs) \) is continuous.
    Finally, it is clear that \(\FTO\) and \(\FTO^{-1}\) are each other's inverses.
\end{proof}

\begin{remark}
    By inspecting the proof of \cref{thm:characterization-coefficients-smooth}, one finds that there are \(C_1,C_2 >0\) such that for all \(p \in \csn(\Vecs)\)
    \begin{flalign}
        \label{equ:estimate-HS-sup}
        \pHSnorm(\FT f (\IrredRep)) 
        \leq C_{1}\inf_{j \in \N} \IrredRepEP^{-j+\Dim} \psupnorm(\LaplP^{j}f), && f \in \Smooth(\LieG;\Vecs),\ \IrredRep \in \DLieG, \\
        \label{equ:estimate-sup-HS}
        \psupnorm(\LaplP^{j}\mathcal{F}^{-1}(T)) \leq C_{2} \sup_{\IrredRep \in \DLieG} \IrredRepEP^{j+\Dim}\pHSnorm(T_{\IrredRep}), && T = (T_{\IrredRep})_{\IrredRep} \in \SmoothC(\DLieG,\Vecs),\ j \in \N
        .
    \end{flalign}
\end{remark}

Given \(h >0\), we define \(\SDiffC{h}(\DLieG;\Vecs)\) as the space consisting of all \(T = (T_{\IrredRep})_{\IrredRep} \in \prod_{\IrredRep \in \DLieG} \Vecs \otimes \mathcal{L}(\IrredRepS)\) such that, for all \(p \in \csn(\Vecs)\),
\[
    \FT{p}_{\Weight,h}(T) 
    = \sup_{\IrredRep \in \DLieG} \pHSnorm(T_{\IrredRep}) e^{\frac{1}{h}\Weight\left(\sqrt{\IrredRepE}\right)} 
    < \infty
    .
\]
We endow this space with the Hausdorff locally convex topology generated by the system of seminorms \(\{ \FT{p}_{\Weight,h} \mid p \in \csn(\Vecs) \}\).
We set 
\[
    \BDiffC(\DLieG;\Vecs) = \varprojlim_{h \to 0^+} \SDiffC{h}(\DLieG;\Vecs)
    \qquad \text{and} \qquad
    \RDiffC(\DLieG;\Vecs) = \varprojlim_{\Weight* = o(\Weight)} \BDiffC[\Weight*](\DLieG;\Vecs) 
    .
\]
Condition \((\gamma)\) implies that \(\UDiffC(\DLieG;\Vecs)\) is continuously embedded into \(\SmoothC(\DLieG;\Vecs)\).
We define \(\BUDiffC(\DLieG;\Vecs) = \bigcup_{B \in \Bnds} \UDiffC(\DLieG;\Vecs_{B})\) and endow this space with its natural convex vector bornology.

For $h >0$ we define $\SDiff{h}_{\LaplP*}(\LieG;\Vecs)$ as the space consisting of all $f \in C^\infty(G; \Vecs)$ such that 
    $p_{\Weight,h}(f) < \infty$ for all $p \in \csn(E)$. We endow $\SDiff{h}_{\LaplP*}(\LieG;\Vecs)$ with the Hausdorff locally convex topology generated by the system of seminorms \(\{ p_{\Weight,h} \mid p \in \csn(\Vecs) \}\). Note that 
    \begin{equation}
        \label{proj-description}
        \BDiff_{\LaplP*}(\LieG;\Vecs) = \varprojlim_{h \to 0^+} \SDiff{h}_{\LaplP*}(\LieG;\Vecs) 
    \end{equation}
as lcHs.

\begin{lemma}~
    \label{thm:characterization-coefficients-ultradifferentiable-step}
    \begin{enumerate}
        \item For all \(h > 0\) there is \(h'<h\) such that \(\FTO \colon \SDiff{h'}_{\LaplP*}(\LieG;\Vecs) \to \SDiffC{h}(\DLieG;\Vecs)\) is continuous.
        \item For all \(h > 0\) there is \(h'<h\) such that \(\FTO^{-1} \colon \SDiffC{h'}(\DLieG;\Vecs) \to \SDiff{h}_{\LaplP*}(\LieG;\Vecs)\) is continuous.
    \end{enumerate}
\end{lemma}
\begin{proof}
    \((i)\) 
    Fix \(h'' < h\).
    Property \eqref{equ:estimate-weight-young} implies that there are \(h' < h''\) and \(C >0\) such that
    \[
        \frac{1}{h''} \Weight(t)
        \leq \sup_{j \in \N} 2j \log(t) - \frac{1}{h'} \WeightYoung\left(2jh'\right) + \log C
        , \qquad t > 0
        .
    \]
    In view of condition \((\gamma)\), there is \(C' >0\) such that
    \[
        \IrredRepEP^{\Dim} e^{\frac{1}{h} \Weight\left(\sqrt{\IrredRepE}\right)} \leq C'e^{\frac{1}{h''} \Weight\left(\sqrt{\IrredRepEP*}\right)}
        , \qquad \IrredRep \in \DLieG
        .
    \]
    Let \(f\in \SDiff{h'}_{\LaplP*}(\LieG;\Vecs)\) be arbitrary.
    By \eqref{equ:estimate-HS-sup}, we obtain that, for all \(p \in \csn(E)\) and \(\IrredRep \in \DLieG\),
    \begin{align*}
        \pHSnorm(\FT f(\IrredRep))
         &\leq C_{1}\inf_{j \in \N} \IrredRepEP^{-j+\Dim} \psupnorm(\LaplP^{j}f) \\
         &\leq C_{1}p_{\Weight,h'}(f) \IrredRepEP^{\Dim}\inf_{j \in \N} \IrredRepEP^{-j} e^{\frac{1}{h'}\WeightYoung(2jh')} \\
         &\leq CC_{1}p_{\Weight,h'}(f) \IrredRepEP^{\Dim}e^{-\frac{1}{h''} \Weight\left(\sqrt{\IrredRepEP*}\right)} \\
         &\leq CC'C_{1}p_{\Weight,h'}(f) e^{-\frac{1}{h} \Weight\left(\sqrt{\IrredRepE}\right)}
         .
    \end{align*}
    Hence,
    \[
        \FT{p}_{\Weight,h}(\widehat{f}) \leq CC'C_1p_{\Weight,h'}(f)
        .
    \]
    \((ii)\) By conditions \((\alpha)\) and \((\gamma)\), there are \(h' < h\) and \(C>0\) such that
    \[
        \IrredRepEP^{\Dim} e^{\frac{1}{h} \Weight\left(\sqrt{\IrredRepEP*}\right)} \leq Ce^{\frac{1}{h'} \Weight\left(\sqrt{\IrredRepE}\right)}
        , \qquad \IrredRep \in \DLieG
        .
    \]
    Let \(T = (T_{\IrredRep})_{\IrredRep} \in \SDiffC{h'}(\DLieG;\Vecs)\) be arbitrary.
    The definition of the Young conjugate \(\WeightYoung\) and \eqref{equ:estimate-sup-HS} imply that, for all \(p \in \csn(E)\) and \(j \in \N\),
    \begin{align*}
        \psupnorm(\LaplP^{j}\mathcal{F}^{-1}(T)) 
         &\leq C_{2} \sup_{\IrredRep \in \DLieG} \IrredRepEP^{j+\Dim}\pHSnorm(T_{\IrredRep}) \\
         & \leq C_{2} \widehat{p}_{\Weight,h'}(T) \sup_{\IrredRep \in \DLieG}\IrredRepEP^{j+\Dim}e^{-\frac{1}{h'} \Weight\left(\sqrt{\IrredRepE}\right)}\\
         & \leq CC_{2} \widehat{p}_{\Weight,h'}(T) \sup_{\IrredRep \in \DLieG}\IrredRepEP^{j}e^{-\frac{1}{h} \Weight\left(\sqrt{\IrredRepEP*}\right)} \\
         & \leq CC_{2} \widehat{p}_{\Weight,h'}(T) e^{\frac{1}{h} \WeightYoung\left(2jh\right) } 
         .
    \end{align*}
    Hence,
    \[
        {p}_{\Weight,h}(\mathcal{F}^{-1}(T)) \leq CC_{2}\widehat{p}_{\Weight,h'}(T) 
        .
        \qedhere
    \]
\end{proof}

\begin{proposition}
    \label{thm:characterization-coefficients-ultradifferentiable}
    The Fourier transform \(\FTO \colon \UDiff(\LieG;\Vecs) \to \UDiffC(\DLieG;\Vecs)\) is an isomorphism of lcHs and \(\FTO^{-1}\) is its inverse.
\end{proposition}
\begin{proof}
    The continuity of the maps \(\FTO \colon \BDiff(\LieG;\Vecs) \to \BDiffC(\DLieG;\Vecs)\) and \(\FTO^{-1} \colon \BDiffC(\DLieG;\Vecs) \to \BDiff(\LieG;\Vecs)\) follows from \cref{thm:ultradifferentiability-laplacian}, the equality \eqref{proj-description}, and \cref{thm:characterization-coefficients-ultradifferentiable-step}.
    The corresponding statements in the Roumieu case follow from those in the Beurling case by the definition of the spaces \(\RDiff(\LieG;\Vecs) \) and \(\RDiffC(\DLieG;\Vecs)\).
\end{proof}
\cref{thm:characterization-coefficients-ultradifferentiable} implies the following result.

\begin{corollary} \label{thm:characterization-coefficients-bornological}
    The Fourier transform \(\FTO \colon \BUDiff(\LieG;\Vecs) \to \BUDiffC(\DLieG;\Vecs)\) is a linear isomorphism and \(\FTO^{-1}\) is its inverse.
    Furthermore, both maps are bounded.
\end{corollary}

Finally, we establish some connections between spaces of Beurling and Roumieu type.
These will be used in the next section.
We begin with the following lemma.

\begin{lemma}
    \label{thm:bigger-weight-function}
    Let \(\Weight\) be a weight function satisfying \((\beta_0)\) and let \(g \colon \rinterval{0}{\infty} \to \rinterval{0}{\infty}\) be a function such that \(\Weight(t) = o(g(t))\).
    Then, there exists a weight function \(\Weight*\) that also satisfies \((\beta_0)\) such that \(\Weight(t) = o(\Weight*(t))\) and \(\Weight*(t) = o(g(t))\).
\end{lemma}
\begin{proof}
    This can be shown by adapting the proof of \cite[Lemma 1.6]{BMT}.
    More precisely, one should change the condition \cite[Eq. (1)]{BMT} on the sequence \((x_n)\) from \(\int_{x_{n+1}}^{\infty} (1+t^{2})^{-1} \Weight(t) d t\leq (n+1)^{-3}\) to \(\Weight(t) \leq (n+1)^{-2} \min(t,g(t))\) for all \(t \geq x_{n+1}\).
    The details are left to the reader.
\end{proof}

\begin{remark}\label{rem:smooth-weight}
    The following variant of \Cref{thm:bigger-weight-function} holds for \(\Weight(t) = \log(1+t)\): 
    \emph{ Let \(g \colon \rinterval{0}{\infty} \to \rinterval{0}{\infty}\) be a function such that \(\log(1+t) = o(g(t))\). Then, there exists a weight function \(\Weight*\) such that \(\log(1+t) = o(\Weight*(t))\) and \(\Weight*(t) = o(g(t))\). }
\end{remark}

\begin{lemma}
    \label{thm:beurling-to-roumieu}
    A set \(\tilde{B} \subset \BSmooth(\LieG;\Vecs)\) is contained and bounded in \(\BBDiff(\LieG;\Vecs)\) if and only if there exists a weight function \(\Weight*\) with \(\Weight(t) = o(\Weight*(t))\) such that \(\tilde{B}\) is contained and bounded in \(\BRDiff[\Weight*](\LieG;\Vecs)\).
\end{lemma}
\begin{proof}
    It suffices to show that for every bounded set \(\tilde{B} \subseteq \BBDiff(\LieG;\Vecs)\) there is a weight function \(\Weight*\) with \(\Weight(t) = o(\Weight*(t))\) such that \(\tilde{B}\) is contained and bounded in \(\BRDiff[\Weight*](\LieG;\Vecs)\), as the other implication is clear.
    By \cref{thm:characterization-coefficients-bornological}, we obtain that \(\FTO (\tilde{B})\) is bounded in \(\BDiffC_b(\LieG;\Vecs)\).
    Hence, there is \(B \in \Bnds\) such that for all \(n \in \N\)
    \[
        C_{n} 
        = \sup_{f \in \tilde{B}}\sup_{\IrredRep \in \DLieG} \HSnorm{\FT f(\IrredRep)}^{B} e^{n \Weight\left(\sqrt{\IrredRepE}\right)} 
        < \infty
        .
    \]
    We define \(g(t) = \max_{n \leq t} n\Weight(t) - \log C_{n}\) for \(t \geq 0\).
    Note that \(\Weight(t) = o(g(t))\).
    Let \(\Weight*\) be as in \cref{thm:bigger-weight-function}.
    Then, \(\FTO(\tilde{B})\) is contained and bounded in \(\RDiffC[\Weight*](\DLieG;\Vecs_B)\).
    The result follows from another application of \cref{thm:characterization-coefficients-bornological}.
\end{proof}

\begin{lemma}
    \label{thm:smooth-to-roumieu}
    For every bounded set \(\tilde{B} \subset \BSmooth(\LieG;\Vecs)\) there exists a weight function \(\Weight*\)  such that \(\tilde{B}\) is contained and bounded in \(\BRDiff[\Weight*](\LieG;\Vecs)\).
\end{lemma}
\begin{proof}
    By \cref{thm:characterization-coefficients-smooth}, there is \(B \in \Bnds\) such that \(\FTO (\tilde{B})\) is contained and bounded in \(\SmoothC(\DLieG;\Vecs_B)\). Hence, for all \(n \in \N\)
    \[
        C_{n} 
        = \sup_{f \in \tilde{B}}\sup_{\IrredRep \in \DLieG} \HSnorm{\FT f(\IrredRep)}^{B} (1+\IrredRepE)^n =
        \sup_{f \in \tilde{B}}\sup_{\IrredRep \in \DLieG} \HSnorm{\FT f(\IrredRep)}^{B} e^{n \log\left(1+\IrredRepE\right)} 
        < \infty.
    \]
    We define \(g(t) = \max_{n \leq t} n\log(1+t) - \log C_{n}\) for \(t \geq 0\).
    Note that \(\log(1+t) = o(g(t))\).       Let \(\Weight*\) be as in \cref{rem:smooth-weight}.   Then, \(\FTO(\tilde{B})\) is contained and bounded in \(\RDiffC[\Weight*](\DLieG;\Vecs_B)\).
    The result now follows from \cref{thm:characterization-coefficients-bornological}.
\end{proof}

We have the following variant of \cref{thm:bigger-weight-function}.

\begin{lemma}[{\cite[Lemma 1.7 and Remark 1.8(1)]{BMT}}]
    \label{thm:smaller-weight-function}
    Let \(\Weight\) be a weight function and let \(g \colon \rinterval{0}{\infty} \to \rinterval{0}{\infty}\) be a function such that \(g(t)=o(\Weight(t))\).
    Then, there exists a weight function \(\Weight*\) such that \(g(t) = o(\Weight*(t))\) and \(\Weight*(t) = o(\Weight(t))\).
\end{lemma}

\begin{lemma}
    \label{thm:roumieu-inductive}
    A set \(\tilde{B} \subset \SmoothC(\DLieG;\Vecs)\) is contained and bounded in \(\BRDiffC(\DLieG;\Vecs)\) if and only if there exist \(h > 0\) and \(B \in \Bnds\) such that \(\tilde{B}\) is contained and bounded in \(\SDiffC{h}(\DLieG;\Vecs_{B})\).
\end{lemma}
\begin{proof}
    It suffices to show that for every bounded set \(\tilde{B} \subseteq \BRDiffC(\DLieG;\Vecs)\) there exist \(h > 0\) and \(B \in \Bnds\) such that \(\tilde{B}\) is contained and bounded in \(\SDiffC{h}(\DLieG;\Vecs_{B})\), as the other implication is trivial.
    There is \(B \in \Bnds\) such that \(\tilde{B}\) is contained and bounded in \(\RDiffC(\DLieG;\Vecs_B)\).
    Suppose that for all \(h >0\) the set \(\tilde{B}\) is not contained and bounded in \(\SDiffC{h}(\DLieG;\Vecs_{B})\).
    Then, there is a sequence \((\IrredRep_{n})_{n \in \N} \subset \DLieG \) with \(\IrredRepE[\IrredRep_{n}] \nearrow \infty\) such that
    \[
        \sup_{ T = (T_{\IrredRep})_{\IrredRep} \in \tilde{B}} \HSnorm{T_{\IrredRep_{n}}}^{B} e^{\frac{1}{n}\Weight\left(\sqrt{\IrredRepE[\IrredRep_{n}]}\right)}
        \geq 1
        .
    \]
    Now define \(g(t) = \max\{ 0, - \log \sup_{ T = (T_{\IrredRep})_{\IrredRep} \in \tilde{B}} \HSnorm{T_{\IrredRep_{n}}}^{B} \}\) for \(\sqrt{\IrredRepE[\IrredRep_{n}]} \leq t < \sqrt{\IrredRepE[\IrredRep_{n+1}]}\).
    Note that \(g(t) = o(\Weight(t))\).
    Let \(\Weight*\) be as in \cref{thm:smaller-weight-function}.
    Then,
    \[
        \sup_{T = (T_{\IrredRep})_{\IrredRep} \in \tilde{B}} \sup_{\IrredRep \in \DLieG} \HSnorm{(T_{\IrredRep})_{\IrredRep}}^{B} e^{\Weight*\left(\sqrt{\IrredRepE}\right)}
        \geq \sup_{n \in \N} \sup_{T = (T_{\IrredRep})_{\IrredRep} \in \tilde{B}} \HSnorm{T_{\IrredRep_{n}}}^{B}e^{g\left(\sqrt{\IrredRepE[\IrredRep_{n}]}\right)} e^{\Weight*\left(\sqrt{\IrredRepE[\IrredRep_{n}]}\right) - g\left(\sqrt{\IrredRepE[\IrredRep_{n}]}\right)}
        = \infty
        ,
    \]
    which contradicts the assumption that \(\tilde{B}\) is bounded in \(\RDiffC(\DLieG;\Vecs_{B})\). 
\end{proof}

\section{Proof of the main result}
\label{sec:main-proof}

We have done all the necessary work in preparation to show \cref{thm:factorization-vectors}.
Our proof is based on the following lemma, which may be considered as the bounded strong factorization property for the \(\Vecs\)-valued discrete space \(\BRDiffC(\DLieG;\Vecs)\).

\begin{lemma}
    \label{thm:factorization-functions-roumieu}
    Let \(\tilde{B} \subset \BRDiffC(\DLieG;\Vecs)\) be bounded. For every \(\IrredRep \in \DLieG\) there is \(C_{\IrredRep} > 0\) such that
    \begin{enumerate}
        \item \((C^{-1}_{\IrredRep}\Id_{\IrredRepS})_{\IrredRep} \in \RDiffC(\DLieG)\);
        \item \(\{(C_{\IrredRep} T_{\IrredRep})_{\IrredRep} \mid (T_{\IrredRep})_{\IrredRep} \in \tilde{B}\}\) is contained and bounded in \(\BRDiffC(\DLieG;\Vecs)\).
    \end{enumerate}
\end{lemma}
\begin{proof}
    Employing \cref{thm:roumieu-inductive}, we find \(B \in \Bnds\) and \(h > 0\) such that \(\tilde{B}\) is contained and bounded in \(\SDiffC{h}(\DLieG;\Vecs_B)\).
    Fix \(h' > h\) and define \(C_{\IrredRep} = e^{\frac{1}{h'}\Weight\left(\sqrt{\IrredRepE}\right)}\) for \(\IrredRep \in \DLieG\).
    Let \(h'' > h'\).
    By \eqref{equ:summability-eigenvalues} and the condition \((\gamma)\), there is \(C >0\) such that
    \[
        \HSnorm{C_{\IrredRep}^{-1}\Id_{\IrredRepS}} 
        \leq Ce^{-\frac{1}{h''}\Weight\left(\sqrt{\IrredRepE}\right)}
        .
    \]
    Thus, \( (C_{\IrredRep}^{-1}\Id_{\IrredRepS})_{\IrredRep} \in \RDiffC(\DLieG)\).
    Furthermore, we have, for all \(T = (T_{\IrredRep})_{\IrredRep} \in \tilde{B}\)
    \[
        \HSnorm{(C_{\IrredRep} T_{\IrredRep})_{\IrredRep}}^{B}
        = \HSnorm{T_{\IrredRep}}^{B} e^{\frac{1}{h'}\Weight\left(\sqrt{\IrredRepE}\right)} 
        \leq \sup_{T \in \tilde{B}} \norm{T}_{\Weight,h}^{B} e^{-\left(\frac{1}{h} - \frac{1}{h'}\right)\Weight\left(\sqrt{\IrredRepE}\right)}
        .
    \]
    Hence, by another application of \cref{thm:roumieu-inductive}, \(\{(C_{\IrredRep} T_{\IrredRep})_{\IrredRep} \mid (T_{\IrredRep})_{\IrredRep} \in \tilde{B}\}\) is contained and bounded in \(\BRDiffC(\DLieG;\Vecs)\).
\end{proof}

Given a representation \(\Rep\) of \(\LieG\) on \(\Vecs\), we can lift it to the representation \(\Rep\otimes \Id_{\mathcal{L}(\IrredRepS)}\) on \(\Vecs \otimes \mathcal{L}(\IrredRepS)\).
Clearly, 
\[
    \Rep(x)(\Trace[A])
    = \Trace[(\Rep(x)\otimes \Id_{\mathcal{L}(\IrredRepS)}(A)]
    , \qquad A\in E\otimes \mathcal{L}(\IrredRepS)
    .
\]

\begin{proof}[{Proof of \cref{thm:factorization-vectors}}]
    By \cref{thm:beurling-to-roumieu}, it is sufficient to consider the Roumieu case.

    Let \(\tilde{B} \subset \BRDiffV\) be bounded.
    By \cref{thm:convolution-to-composition}, \cref{thm:characterization-coefficients-bornological}, and \cref{thm:factorization-functions-roumieu}, there exist \(g \in \BRDiff(\LieG)\) and constants \(C_{\IrredRep} > 0\), \(\IrredRep \in \DLieG\), such that
    \begin{enumerate}
        \item \(\tilde{B}' = \{ f_v \mid v \in \tilde{B}\}\) is contained and bounded in \(\BRDiff(\LieG;\Vecs)\), where the function \(f_v\) is defined via \(\FT {f_v}(\IrredRep) = C_\IrredRep \FT{\Orbit{v}}(\IrredRep)\), \(\IrredRep \in \DLieG\).
        \item \(\Orbit{v} = g \ast f_v\) for all \(v \in \tilde{B}\).
    \end{enumerate}
    Set \(\tilde{v} = f_v(e)\) for \(v \in \tilde{B}\).
    It suffices to show that \(f _v= \gamma_{\tilde{v}}\), as this implies that \(\{\tilde{v} \mid v \in \tilde{B}\}\) is contained and bounded in \(\BRDiffV\)
    and 
    \[
        v 
        = \Orbit{v}(e) 
        = g \ast \gamma_{\tilde{v}}(e) 
        = \Action(\check{g})(\tilde{v})
        , \qquad v \in \tilde{B}
        .
    \]
    Using the Fourier inversion formula \eqref{equ:fourier-inversion}, we infer that
    \[
        f_v(x) 
        = \sum_{\IrredRep \in \DLieG} \IrredRepD \Trace [ \IrredRep^{\ast}(x)\circ \FT {f_v}(\IrredRep)] 
        = \sum_{\IrredRep \in \DLieG} \IrredRepD C_{\IrredRep} \Trace [ \IrredRep^{\ast}(x)\circ \FT{\Orbit{v}} (\IrredRep) ]
        , \qquad x \in G
        .
    \]
    Setting \(x = \Origin\), we get
    \[
        \tilde{v}
        = \sum_{\IrredRep \in \DLieG} \IrredRepD C_{\IrredRep} \Trace [\FT {\Orbit{v}} (\IrredRep)] 
        .
    \]
    Noticing that
    \begin{align*}
        (\Rep(x)\otimes \Id_{\mathcal{L}(\IrredRepS)}) (\FT {\Orbit{v}} (\IrredRep))
         &= \int _{\LieG} (\Rep(x)\otimes \Id_{\mathcal{L}(\IrredRepS)})(\Orbit{v}(y) \otimes \IrredRep (y)) \dl{y}
         = \int _{\LieG} \Orbit{v}(y) \otimes \IrredRep (x^{-1}y) \dl{y} \\
         &= \IrredRep^{\ast}(x)\circ \FT{\Orbit{v}} (\IrredRep)
         ,
    \end{align*}
    we obtain
    \begin{align*}
        \Orbit{\tilde{v}}(x)
         &= \Rep(x)\left(\sum_{\IrredRep \in \DLieG} \IrredRepD C_{\IrredRep} \Trace [\FT {\Orbit{v}} (\IrredRep)] \right ) 
         = \sum_{\IrredRep \in \DLieG} \IrredRepD C_{\IrredRep} \Trace [(\Rep(x)\otimes \Id_{\mathcal{L}(\IrredRepS)}) (\FT {\Orbit{v}} (\IrredRep))] \\
         &= \sum_{\IrredRep \in \DLieG} \IrredRepD C_{\IrredRep} \Trace [ \IrredRep^{\ast}(x)\circ \FT{\Orbit{v}} (\IrredRep)] 
         = f_{v}(x)
         , \qquad x\in G
         .
         \qedhere
    \end{align*}
\end{proof}

\section{Factorization of non-quasianalytic vectors}\label{sec: factorization non-quasianalytic vectors}

Recall that \(\Weight\) is non-quasianalytic if \( \displaystyle \int _{0}^{\infty} t^{-2}\Weight(t)\dl{t} < \infty\). 
This is equivalent to the existence of non-zero compactly supported elements in \(\UDiff(M)\), which in turn yields the possibility of using partitions of unity from this class \cite{BMT}.
We have the following improvements to \cref{thm:factorization-vectors} for non-quasianalytic vectors and \cref{thm:factorization-vectors-smooth} (part \((ii)\) recovers a result of Dixmier and Malliavin for Fréchet representations \cite[Th\'eor\`eme 4.11]{D-M}).

\begin{theorem}
    \label{thm:factorization-vectors-non-quasianalytic}
    Let \(\Rep\) be a  representation of \(\LieG\) acting on \(\Vecs\) and let \(V\) be a neighborhood of the identity \(\Origin\) in \(\LieG\).
    \begin{enumerate}
        \item Suppose that \(\Weight\) is non-quasianalytic.
            Given a bounded subset \(B\) of \(\BUDiffV\), there are \(f\in \UDiff(\LieG)\) with \(\Supp f\subseteq V\) and a bounded subset \(B'\) of \(\BUDiffV\) such that \(B=\Action(f)(B')\).
        \item Given a bounded subset \(B\) of \(\SmoothV\), there are \(f\in C^{\infty}(\LieG)\) with \(\Supp f\subseteq V\) and a bounded subset \(B'\) of \(\SmoothV\) such that \(B=\Action(f)(B')\).
    \end{enumerate}
\end{theorem}

Part \((ii)\) of \cref{thm:factorization-vectors-non-quasianalytic} follows from \((i)\) by employing \cref{thm:smooth-to-roumieu}, while part \((i)\) can be established in the same way as in \cref{sec:main-proof} if we replace \cref{thm:factorization-functions-roumieu} by the following one: 

\begin{lemma}
    \label{thm:factorization-functions-roumieu-supportversion}
    Let \(V\) be a neighborhood of the identity \(\Origin\) in \(\LieG\) and suppose that \(\Weight\) is non-quasianalytic. 
    For every bounded set \(\tilde{B} \subset \BRDiffC(\DLieG;\Vecs)\) there is \(S=(S_{\IrredRep})_{\IrredRep} \in \RDiffC(\DLieG)\)
    such that
    \begin{enumerate}
        \item \(\Supp(\FTO^{-1}(S))\subset V\);
        \item each \(S_{\IrredRep}\) is invertible;  
        \item \(\{(S^{-1}_{\IrredRep}\circ T_{\IrredRep})_{\IrredRep} \mid (T_{\IrredRep})_{\IrredRep} \in \tilde{B}\}\) is contained and bounded in \(\BRDiffC(\DLieG;\Vecs)\).
    \end{enumerate}
\end{lemma}
\begin{proof} We adapt an idea of Dixmier and Malliavin \cite[Lemme 4.5]{D-M} used by them for the one-dimensional torus. 
    Let \(W\) be an inverse symmetric neighborhood of \(e\) such that \(W\cdot W\subseteq V\). 
    Choose \(B \in \Bnds\) and \(h\) exactly as in the proof of \cref{thm:factorization-functions-roumieu} and fix \(h'>h\).
    Note that \(\FTO^{-1}( ( e^{-\frac{1}{2h'}\Weight\left(\sqrt{\IrredRepE}\right)}\Id_{\IrredRepS})_{\IrredRep})\in  \RDiff(\LieG)\). 
    Using the compactness of \(\LieG\) and the non-quasianalyticity of \(\Weight\), we can write (where \(k\) just depends on \(V\)) 
    \begin{equation}\label{eq: Fourier inverse function for D-M localization trick}
        \FTO^{-1}( ( e^{-\frac{1}{2h'}\Weight\left(\sqrt{\IrredRepE}\right)}\Id_{\IrredRepS})_{\IrredRep}) 
        = \sum_{j=1}^{k}\psi_{j}
        ,
    \end{equation}
    where each \(\psi_{j}\) \(\in \RDiff(\LieG)\) is supported in a right-translation of \(W\).
    Consider
    \[
        f 
        = \sum_{j=1}^{k}\psi_{j}^{\ast} \ast \psi_{j}
        ,
    \]
    \(\psi_{j}^{\ast}(x)=\overline{\psi_{j}(x^{-1})}\), which is supported in \(V\).
    We have 
    \[ 
        S
        = (S_{\IrredRep})_{\IrredRep} 
        = (\FT{f}({\IrredRep}))_{\IrredRep} \in\RDiffC(\DLieG)
        ,
    \]
    so \((i)\) is fulfilled.
    Moreover, it is also clear that each \(\FT{f}({\IrredRep})\) is positive semi-definite. 
    Let \(\mu_{\IrredRep}\) be the smallest eigenvalue of \(\FT{f}({\IrredRep})\). 
    We need a lower bound for \(\mu_{\IrredRep}\). 
    If \(v\) is a non-zero eigenvector associated with the eigenvalue \(\mu_{\IrredRep}\), we have
    \[
        \mu_{\IrredRep}
        = \frac{(\FT{f}({\IrredRep})v, v)_{\IrredRepS}}{\norm{v}_{\IrredRepS}^2}
        = \frac{1}{\norm{v}_{\IrredRepS}^2} \sum_{j=1}^{k} \norm{\FT{\psi}_{j}(\IrredRep)v}_{\IrredRepS}^{2}
        \geq \frac{1}{k\norm{v}_{\IrredRepS}^2}\norm*{ \sum_{j=1}^{k}\FT{\psi}_{j}(\IrredRep)v }_{\IrredRepS}^{2}=\frac{ e^{-\frac{1}{h'}\Weight\left(\sqrt{\IrredRepE}\right)}}{k}
        \: ,
    \]
    where we have used \eqref{eq: Fourier inverse function for D-M localization trick}. 
    This implies that \(S_\IrredRep\) is positive definite and thus, in particular, invertible. 
    Hence, \(S\) satisfies \((ii)\). 
    Fix \(h<h''<h'\). 
    Then,
    \begin{equation}\label{eq: HS of inverse S}
        \HSnorm{S^{-1}_{\IrredRep}}
        \leq \frac{\IrredRepD^{1/2}}{\mu_{\IrredRep}} 
        \leq k \IrredRepD^{1/2}e^{\frac{1}{h'}\Weight\left(\sqrt{\IrredRepE}\right)}
        \leq Ce^{\frac{1}{h''}\Weight\left(\sqrt{\IrredRepE}\right)}
        , 
    \end{equation}
    for some \(C>0\), as one infers from condition \((\gamma)\) and \eqref{equ:summability-eigenvalues}.
    Part \((iii)\) now follows from \eqref{eq: HS of inverse S} and the Cauchy-Schwarz inequality for Hilbert-Schmidt norms, which allow us to deduce that \(\{(S^{-1}_{\IrredRep}\circ T_{\IrredRep})_{\IrredRep} \mid (T_{\IrredRep})_{\IrredRep} \in \tilde{B}\}\) is contained and bounded in \(\SDiffC{\frac{h''h}{h''-h}}(\DLieG;\Vecs_{B})\).
\end{proof}


\begin{thebibliography}{10}

    \bibitem{B-D} 
    J.~Bonet, P.~Domanski, \emph{Real analytic curves in Fr\'{e}chet spaces and their duals}, Monatsh. Math. \textbf{126} (1998), 13--36.

    \bibitem{BJM} 
    J.~Bonet, E.~Jord\'a, M. Maestre, {\em Vector-valued meromorphic functions}, Arch. Math. \textbf{79} (2002), 353--359.

    \bibitem{Bierstedt} 
    K.~D.~Bierstedt, \emph{An introduction to locally convex inductive limits}, pp.\ 35--133, in: \emph{Functional analysis and its applications} (Nice, 1986), World Sci. Publishing, Singapore, 1988.

    \bibitem{BMT} 
    R.~W.~Braun, R.~Meise, B.~A.~Taylor, \emph{Ultradifferentiable functions and Fourier analysis}, Results Math. \textbf{17} (1990), 206--237.


    \bibitem{Cohen59}
    P.~J.~Cohen, \emph{Factorization in group algebras}, Duke Math. J. \textbf{26} (1959), 199-205.

    \bibitem{D-H-V} A.~Debrouwere, M.~Huttener, J.~Vindas, \emph{Quasinormability and property \((\Omega)\) for spaces of smooth and ultradifferentiable vectors associated with Lie group representations,} preprint, arXiv:2403.08296.

    \bibitem{DPV21}
    A.~Debrouwere, B.~Prangoski, J.~Vindas, \emph{Factorization in Denjoy-Carleman classes associated to representations of \((\mathbb{R}^d,+)\)}, J. Funct. Anal. \textbf{280} (2021), 108831.

    \bibitem{DPV22}
    A.~Debrouwere, B.~Prangoski, J.~Vindas, \emph{On the projective description of spaces of ultradifferentiable functions of Roumieu type}, in: \emph{Current trends in Analysis, its Applications and Computation}. Proceedings of the 12th ISAAC congress (Aveiro, 2019), pp. 363--372. Trends in Mathematics, Birkh\"{a}user, Cham, 2022.

    \bibitem{D-M} 
    J.~Dixmier, P.~Malliavin, \emph{Factorisations de fonctions et de vecteurs indéfiniment différentiables}, Bull. Sci. Math. \textbf{102} (1978), 307--330.

    \bibitem{ehrenpreis}
    L.~Ehrenpreis, \textit{Solution of some problems of division. IV. Invertible and elliptic operators}, Amer. J. Math. \textbf{82} (1960), 522--588.

    \bibitem{Fegan91}
    H.~D.~Fegan, \emph{Introduction to compact Lie groups}, World Sci. Publ., 1991.

    \bibitem{Folland95}
    G.~B.~Folland, \emph{A course in abstract harmonic analysis}, CRC Press, 1995.

    \bibitem{FollandPDE}
G.~B.~Folland, \emph{Introduction to partial differential equations,} Princeton University Press, 1996.

    \bibitem{FS}
    S.~F\"urd\"os, G.~Schindl, \emph{The theorem of iterates for elliptic and non-elliptic operators}, J. Funct. Anal. \textbf{283} (2022), 109554.

    \bibitem{G-K}    
    H.~Gimperlein, B.~Kr\"{o}tz,  \emph{On Sobolev norms for Lie group representations}, J. Funct. Anal. \textbf{280} (2021), 108882.

    \bibitem{G-K-L} 
    H.~Gimperlein, B.~Kr\"{o}tz, C.~Lienau, \emph{Analytic factorization of Lie group representations}, J. Funct. Anal. \textbf{262} (2012), 667--681.

    \bibitem{G-K-S} 
    H.~Gimperlein, B.~Kr\"{o}tz, H.~Schlichtkrull, \emph{Analytic representation theory of Lie groups: general theory and analytic globalizations of Harish-Chandra modules,} Compos. Math. \textbf{147} (2011), 1581--1607.

    \bibitem{Glo}
    H.~Gl\"{o}ckner, \emph{Continuity of \(LF\)-algebra representations associated to representations of Lie
    groups,} Kyoto J. Math. \textbf{53} (2013), 567--595.

    \bibitem{Goodman1969} 
    R.~Goodman, \emph{Analytic and entire vectors for representations of Lie groups,} Trans. Amer. Math. Soc. \textbf{143} (1969), 55--76.

    \bibitem{Goodman1} 
    R.~Goodman, \emph{Differential operators of infinite order on a Lie group I}, J. Math. Mech, \textbf{19} (1970), 879--894.

    \bibitem{Goodman2} 
    R.~Goodman, \emph{Differential operators of infinite order on a Lie group II}, Indiana Univ. Math. J., \textbf{21} (1970), 383--409.

    \bibitem{GW} 
    R.~Goodmann, N.~R.~Wallach, \emph{Whittaker vectors and conical vectors}, J. Funct. Anal. \textbf{39}, (1980), 199--279.

    \bibitem{Harish-Chandra} 
    Harish-Chandra, \emph{Representations of a semisimple Lie group on a Banach space. I,} Trans. Amer. Math. Soc. \textbf{75} (1953), 185--243.


    \bibitem{Hoermander90}
    L.~Hörmander, \emph{The analysis of linear partial differential operators. {I}. Distribution theory and Fourier analysis}, Springer-Verlag, Berlin, 1990.

    \bibitem{JH} 
    J.~Juan-Huguet, \emph{Non-quasianalytic curves in Fréchet spaces}, Monatsh Math \textbf{164} (2011), 427--437.

    \bibitem{Komatsu3} 
    H.~Komatsu, \emph{Ultradistributions. III. Vector-valued ultradistributions and the theory of kernels}, J. Fac. Sci. Univ. Tokyo Sect. IA Math. \textbf{29} (1982), 653--717.

    \bibitem{KM}
    A. Kriegl and P.~W. Michor, \textit{The convenient setting of global analysis}, Amer. Math. Soc., Providence, RI, 1997.

    \bibitem{MeiseVogtBook} 
    R.~Meise, D.~Vogt, \emph{Introduction to functional analysis,} Clarendon Press, Oxford, 1997.

    \bibitem{Nelson59} 
    E.~Nelson, \emph{Analytic vectors,} Ann. Math. \textbf{70} (1959), 572--615.

    \bibitem{Rainer14}
    A.~Rainer, G.~Schindl, \emph{Composition in ultradifferentiable classes}, Studia Math. \textbf{224} (2014), 97--131.

    \bibitem{rst} 
    L.~A.~Rubel, W.~A.~Squires, B.~A.~Taylor, \textit{Irreducibility of certain entire functions with applications to harmonic analysis}, Ann. of Math. \textbf{108} (1978), 553--567.

    \bibitem{Rudin57}
    W.~Rudin, \emph{Factorization in the group algebra of the real line}, Proc. Nat. Acad. Sci. U.S.A. \textbf{43} (1957), 339--340.

    \bibitem{Salem39}
    R. Salem, \emph{Sur les transformations des s\'eries de Fourier}, Fund. Math. {\bf 33} (1939), 108--114.

    \bibitem{Stein70}
    E.~M.~Stein, \emph{Topics in harmonic analysis, related to the Littlewood-Paley theory}, Princeton Univ. Press, 1970.

    \bibitem{TrevesBook1980II} F.~Tr\`{e}ves, \emph{Introduction to pseudodifferential and Fourier integral operators, Vol. 2.
    Fourier integral operators,} Univ. Ser. Math. Plenum Press, New York-London, 1980.


    \bibitem{yul} 
    R.~S.~Yulmukhametov, \textit{Solution of the L. Ehrenpreis problem on factorization}, Mat. Sb. \textbf{190} (1999), 123--157; translation in Sb. Math. \textbf{190} (1999), 597--629.

\end{thebibliography}
\end{document}